\documentclass[10pt]{article}%
\usepackage{amssymb}
\usepackage{amsfonts}
\usepackage{amsmath}
\usepackage{color}
\usepackage{graphicx}%
\providecommand{\U}[1]{\protect\rule{.1in}{.1in}}
\newtheorem{theorem}{Theorem}

\newtheorem{corollary}[theorem]{Corollary}

\newtheorem{definition}[theorem]{Definition}

\newtheorem{lemma}[theorem]{Lemma}

\newtheorem{proposition}[theorem]{Proposition}
\newtheorem{remark}[theorem]{Remark}

\newenvironment{proof}[1][Proof]{\noindent\textbf{#1.} }{\ \rule{0.5em}{0.5em}}

\numberwithin{equation}{section}
\begin{document}

\title{On a nonlocal Cahn-Hilliard equation with a reaction term
\footnote{{\bf Acknowledgment.}\quad\rm  The financial support of the FP7-IDEAS-ERC-StG \#256872
(EntroPhase) is gratefully acknowledged by the authors. The present paper 
also benefits from the support of the GNAMPA (Gruppo Nazionale per l'Analisi Matematica, la Probabilit\`a e le loro Applicazioni) of INdAM (Istituto Nazionale di Alta Matematica).}}
\author{
{Stefano Melchionna\thanks{
University of Vienna ,
Faculty of Mathematics,
Oskar-Morgenstern-Platz 1, 
1090 Wien, Austria. E-mail: \textit{melchionna.s90@gmail.com}  }}
\and 
{Elisabetta Rocca\thanks{
Weierstrass Institute for Applied Analysis and Stochastics, Mohrenstr. 39, D-10117 Berlin,
Germany.
E-mail: \textit{Elisabetta.Rocca@wias-berlin.de} and Dipartimento di Matematica ``F. Enriques'',
Universit\`{a} degli Studi di Milano, Milano I-20133, Italy.
E-mail: \textit{elisabetta.rocca@unimi.it} }}}
\maketitle

\begin{abstract}
We prove existence, uniqueness, regularity and separation properties for a
nonlocal Cahn-Hilliard equation with a reaction term. We deal here with the
case of logarithmic potential and degenerate mobility as well an uniformly
lipschitz in $u$ reaction  term $g(x,t,u).$

\end{abstract}

\section{Introduction}

Our aim is to generalize existence, uniqueness, separation property and
regularity results, proved by Gajewski, Zacharias \cite{GZ} and Londen and
Petzeltov\`{a} \cite{LP2} for the nonlocal Cahn-Hilliard equation, to the
nonlocal Cahn-Hilliard equation with reaction. Hence, we aim to study the 
 following initial boundary value problem: %
\begin{equation}
\partial_t u-\nabla\cdot(\mu\nabla v)=g(u)\text{ in
}Q\text{,} \label{eq:CH}%
\end{equation}%
\begin{equation}
v=f^{\prime}(u)+w\text{ in }Q\text{,} \label{def:v}%
\end{equation}%
\begin{equation}
w(x.t)=\int_{\Omega}K(\left\vert x-y\right\vert )(1-2u(y,t))dy\text{ for
}(x,t)\in Q\text{,} \label{def:w}%
\end{equation}%
\begin{equation}
n\cdot\mu\nabla v=0~\ \text{on }\Gamma, \label{eq:bc}%
\end{equation}%
\begin{equation}
u(x,0)=u_{0}(x)\text{, }x\in\Omega\text{,} \label{eq:ic}%
\end{equation}
where $Q=\Omega\times(0,T)$, $\Omega\subset\mathbb{R}^d$ is a bounded domain, $\Gamma=\partial\Omega\times(0,T)$, and  $n$ is the
outer unit normal to $\partial\Omega$. The functions $f$ and $\mu$ are
definite by
\begin{equation}
f(u)=u\log u+(1-u)\log(1-u)\text{,} \label{def:f}%
\end{equation}%
\begin{equation}
\mu=\frac{1}{f^{\prime\prime}(u)}=u(1-u)\text{.} \label{def:mu}%
\end{equation}
The man novelty of the paper is that we can take into account in our analysis of the reaction term $g$ in 
\eqref{eq:CH}, which can be taken as a Lipschitz continuous function of the unknown $u$. 

Let us briefly recall here - for the readers' convenience - the derivation 
of the nonlocal Cahn-Hilliard equation and the comparison with the local one. 
System \eqref{eq:CH}--\eqref{eq:ic} describes the evolution of a binary alloy with components $A$ and $B$ occupying a spatial
domain $\Omega$. We denote by $u$ the local concentration of $A$. To describe
phase separation in binary system one uses generally the standard local Cahn-Hilliard
equation, which is derived (cf.~\cite{CH}) from a free energy functional of this form 
of the form
\begin{equation}
E_{CH}(u)=\int_{\Omega}\left(  \frac{\tau^{2}}{2}\left\vert \nabla
u\right\vert ^{2}+F(u)\right)  dx\text{.} \label{CHE}%
\end{equation}
Here $F(u)$ denotes the Helmholtz free energy density of $A$. It is defined
as
\begin{equation}
F(u)=2K_{B}T_{c}u(1-u)+K_{B}Tf(u)\text{,} \label{CHF}%
\end{equation}
where $K_{B}$ is the Boltzmann's constant, $T$ is the system temperature,
$T_{c}$ is called critical temperature and $f$ is defined as
\begin{equation}
f(u)=u\ln u+\left(  1-u\right)  \ln\left(  1-u\right)  \text{.} \label{CHf}%
\end{equation}
Considering that the dynamics tends to minimize the energy $E_{CH}$, Cahn
obtained (\cite{Ca}) the following equation for $u$:%
\begin{equation}
u_{t}+\nabla\cdot J=0 \label{prim}%
\end{equation}
where $J$ is defined as
\begin{equation}\label{jei}
J=-\mu(u)\nabla v\text{.}%
\end{equation}
The function $\mu$ is named mobility and $v$ denotes the first variation of
$E_{CH}$ with respect to $u$:%
\begin{equation}
v=\frac{\delta E_{CH}}{\delta u}=F^{\prime}(u)-\tau^{2}\Delta u\text{,}
\label{ov}%
\end{equation}
known as chemical potential. For simplicity, in literature the mobility is
often chosen constant although its physical (degenerate) relevant form is%
\begin{equation}
\mu=au(1-u)\text{, }a>0 \label{omu}%
\end{equation}
(see \cite{Ca}), where $a$ is a positive function possibly depending on $u$
and $\nabla u$ separated from $0$ (in literature $a$ is often a positive
constant). Equation (\ref{prim}) is, hence, a 4th order nonlinear PDE known as
Cahn-Hilliard equation:%
\begin{equation}
u_{t}+\nabla\cdot\left(  \mu(u)\nabla(F^{\prime}(u)-\tau^{2}\Delta u)\right)
=0\text{,} \label{och}%
\end{equation}
which is usually coupled with the following boundary conditions:%
\begin{equation}
\frac{\partial u}{\partial n}=0\text{ on }\partial\Omega\quad\hbox{and } \mu(u)n\cdot\nabla v=0\text{ on }\partial\Omega\text{.} \label{obc}%
\end{equation}
This last condition ensures the mass conservation. Indeed, thanks to
(\ref{obc}), we have:
\[
\frac{d}{dt}\int_{\Omega}u=\int_{\Omega}u_{t}=-\int_{\Omega}\nabla\cdot\left(
\mu(u)\nabla v\right)  =\int_{\partial\Omega}\mu(u)n\cdot\nabla v=0\text{.}%
\]

Despite numerical results on the Cahn-Hilliard equation are in good agreement
with experiments, G. Giacomin and J. L. Lebowitz in \cite{GL1} and  \cite{GL2}
showed that Cahn-Hilliard equation cannot be derived from microscopic
phenomena. This motivation led G. Giacomin and J. L. Lebowitz to study the
problem of phase separation from the microscopic viewpoint using statistical
mechanics. Then, performing the hydrodynamic limit they deduced a continuum
model. By proceeding in this way they found a nonlocal version of the
Cahn-Hilliard equation that is a second order nonlinear integrodifferential
equation:
\begin{equation}
u_{t}+\nabla\cdot J=0 \label{nlch}%
\end{equation}
where $J$ is defined as in \eqref{jei}, $\mu$ denotes the mobility term (defined as in (\ref{omu})), and
$v=\frac{\delta E}{\delta u}$. Here the energy functional $E$ is given by%
\begin{equation}
E(u)=\frac{1}{2}\int_{\Omega}\int_{\Omega}K(x-y)(u(x)-u(y))^{2}dxdy+\int
_{\Omega}f(u)+ku(1-u)dx\text{.} \label{nle}%
\end{equation}
This leads to
\begin{equation}
v=f^{\prime}(u)+w\text{, where }w=K\ast(1-2u)\text{,} \label{nlv}%
\end{equation}
and where $K$ is a symmetric positive convolution kernel, $k(x)=\int_{\Omega
}K(x-y)dy$ and $f$ is defined as in (\ref{CHf}).

Nonlocal Cahn-Hilliard equation is generally coupled with the boundary
condition%
\begin{equation}\label{bou}
\mu(u)n\cdot\nabla v=0\text{ on }\partial\Omega.
\end{equation}
Thus, the mass-conservation is still preserved. Notice that the nonlocal contribution $\frac{1}{2}\int_{\Omega}\int_{\Omega}K(x-y)(u(x)-u(y))^{2}dxdy$ in 
\eqref{nle}, replacing the local one $\int_\Omega \frac{\tau^{2}}{2}\left\vert \nabla
u\right\vert ^{2}$, better describes the long-range interactions between points in $\Omega$. Moreover, let us note that 
the local term $\int_\Omega \frac{\tau^{2}}{2}\left\vert \nabla
u\right\vert ^{2}$ could be formally obtained from the nonlocal one (cf.~\cite{KRS}).

Adding a reaction term to the Cahn-Hilliard equation is useful in several
applications such as biological models (\cite{KS}), impainting algorithms
(\cite{BEG}), polymers (\cite{BO}). Cahn-Hilliard equation with reaction is
\begin{equation}
u_{t}+\nabla\cdot J=g(u)\text{,} \label{opla}%
\end{equation}
where $J=-\mu\nabla v$ and $v$ as in (\ref{ov}) or as in (\ref{nlv}) and
$g(u)=g(x,t,u)$.

The main difficulties in studying Cahn-Hilliard equation with reaction are due
to the non-conservation of the mass. Indeed, thanks to the boundary condition \eqref{bou},
we have%
\begin{equation}
\frac{d}{dt}\int_{\Omega}u=\int_{\Omega}g\neq0\text{.} \label{massa non cons}%
\end{equation}

Some analytical results on the local Cahn-Hilliard equation with reaction term
are \cite{CMZ2}, \cite{Mi}. Existence and uniqueness for nonlocal
Cahn-Hilliard equation with constant mobility, polynomial potential and
reaction term are proved in \cite{DP}.

To the best of our knowledge no previous works on the nonlocal Cahn-Hilliard
equation with reaction and with singular potential and degenerate mobility
have been proved. Furthermore, our assumptions on the reaction term (see
(\ref{G2})-(\ref{ipo:g5})) are more general then in \cite{CMZ2}, \cite{Mi} and
\cite{DP} and they are satisfied in every application we know (cf., e.g., \cite{KS},
\cite{BEG}, \cite{BO}).

\paragraph{Plan of the paper.} In Section \ref{setting} we set notation, describe assumptions on data and
state the main results. Existence and uniqueness are proved in Section
\ref{dimostra}. Regularity results are proved in Section \ref{RRRR}. Section
\ref{Separazione} is devoted to the proofs of  the separation properties.\ Some remarks are stated
in Section \ref{comm}. Appendix (Section \ref{appendix}) contains example of
convolution kernels and auxiliary theorems.

\section{Assumptions on data and main results\label{setting}}

\subsection{Notation}

Set $\Omega\subset%
\mathbb{R}
^{d}$, $d\in%
\mathbb{N}
$, a bounded domain with a sufficiently smooth boundary (e.g., of class
$C^{1,1}$).

If $X$ is a real Banach space, $X^{\ast}$ will denote its dual. For every
$z\in(H^{1}(\Omega))^{\ast}$ we denote $\bar{z}=\left\langle z,\frac
{1}{\left\vert \Omega\right\vert }\right\rangle $. Here $\left\langle
,\right\rangle $ denotes the pairing of $H^{1}(\Omega)$ and $\left(
H^{1}(\Omega)\right)  ^{\ast}$. Let us introduce also the space $H_{0}%
^{1}(\Omega)=\{z\in H^{1}(\Omega):\bar{z}=0\}$.

Set $H^{1}(0,T,X,X^{\ast})=\{$ $z\in L^{2}(0,T,X):\,z_t\in
L^{2}(0,T,X^{\ast})\}$.

If $z\in$ $H^{1}(0,T,X,X^{\ast})$ the symbols $z^{\prime}$, $\partial_t z$, $\frac{\partial z}{\partial t}$, and $z_{t}$ will denote the partial derivative of $z$ with respect to the
$t$-variable (time). Let $f\in C^{1}(%
\mathbb{R}
)$, we use the symbol $f^{\prime}$ to denote the derivative of the function
$f$. Finally, set $y\in H^{1}([0,T]\times\Omega)$, we indicate the partial
derivative of $y$ with respect to the first variable (time) with the symbols
$\partial_{t}y$ or $\frac{\partial y}{\partial t}$ and the partial derivate of
$y$ with respect to the $x_{i}-$variable with the symbol $\partial_{i}y$.

If $\alpha:%
\mathbb{R}
^{d}\rightarrow%
\mathbb{R}
$ and $\beta:\Omega\subset%
\mathbb{R}
^{d}\rightarrow%
\mathbb{R}
$ are measurable functions $\alpha\ast\beta$ will denote the convolution
product definite by $\alpha\ast\beta(x)=\int_{\Omega}\alpha(x-y)\beta(y)dy$
for $x\in%
\mathbb{R}
^{d}$.

\subsection{Assumptions on data}

The given functions $K$, $u_{0}$ and $g$ will be assumed to fulfill the
following conditions.

\begin{description}
\item[(K)] The convolution kernel $K:%
\mathbb{R}
^{d}\rightarrow%
\mathbb{R}
$ satisfies the assumptions%
\begin{equation}
K(x)=K(-x)\text{ for }a.a.~x\in%
\mathbb{R}
^{d}\text{,} \tag{K1}\label{K1}%
\end{equation}%
\begin{equation}
\sup_{x\in\Omega}\int_{\Omega}\left\vert K(x-y)\right\vert dy<+\infty\text{,}
\tag{K2}\label{K2}%
\end{equation}%
\begin{equation}
\forall p\in\lbrack1,+\infty]\text{ }\exists r_{p}>0\text{ such that
}\left\Vert K\ast\rho\right\Vert _{W^{1,p}(\Omega)}\leq r_{p}\left\Vert
\rho\right\Vert _{L^{p}(\Omega)}\text{,} \tag{K3}\label{K3}%
\end{equation}%
\begin{equation}
\exists C>0\text{ such that }\left\Vert K\ast\rho\right\Vert _{W^{2,2}%
(\Omega)}\leq C\left\Vert \rho\right\Vert _{W^{1,2}(\Omega)}\text{;}
\tag{K4}\label{K4}%
\end{equation}

\item[(u0)] The initial datum $u_{0}$ is supposed to satisfy

\begin{equation}
u_{0}\text{ is measurable, } \tag{U01}\label{u00}%
\end{equation}%
\begin{equation}
0\leq u_{0}(x)\leq1~\text{for a.a.}~x\in\Omega\text{,} \tag{U02}\label{u011}%
\end{equation}%
\begin{equation}
0<\bar{u}_{0}<1\text{;} \tag{U03}\label{u02}%
\end{equation}

\item[(G)] The reaction term $g:\Omega\times%
\mathbb{R}
^{+}\times\lbrack0,1]\rightarrow%
\mathbb{R}
$ 
is such that 
\begin{align}
&g(x,t,s)\text{ is continuous,} \tag{G1}\label{G2}\\
&  \exists L  >0\text{ such that }\left\vert g(x,t,s_{1})-g(x,t,s_{2}%
)\right\vert \leq L\left\vert s_{1}-s_{2}\right\vert 
\forall s_{1},s_{2}  \in\lbrack0,1]\text{, }\forall x\in\Omega,\,\,\forall t\in%
\mathbb{R}
^{+} \tag{G2}%
\label{ipo:gLip}\\
&g(x,t,0)\geq0\geq g(x,t,1)\text{} \quad\forall x\in\Omega\text{, }\forall t\in%
\mathbb{R}^+\,.\label{ipo:g5}\tag{G3}
\end{align}
\end{description}

We remark that, as a consequence of (\ref{G2}) for every $T>0$, there exist
$C>0$ depending on $T$ so that%
\begin{equation}
\left\vert g(x,t,s)\right\vert \leq C\quad \forall s\in\lbrack
0,1],t\in\lbrack0,T],x\in\Omega. \label{ipo:Gi3}%
\end{equation}
Furthermore, as a consequence of (\ref{ipo:gLip}), we have%
\begin{gather}
g\text{ is differentiable for }a.a.\text{~}s\in\lbrack0,1] \,\,  \label{ipo:gUni}
\text{and }\left\vert \partial_{s}g(x,t,s)\right\vert \leq L\text{ for
a.a.}~(x,t,s)\in\Omega\times%
\mathbb{R}
^{+}\times\lbrack0,1],\nonumber
\end{gather}
where $L$ as in (\ref{ipo:gLip}) (see \cite{NZ}).

\begin{remark}
Some examples of convolution kernels $K$ which satisfy the above conditions
(\ref{K1})-(\ref{K4}) are given by Newton potentials:%
\[
\left\{
\begin{array}
[c]{lll}%
K(\left\vert x\right\vert )=k_{d}\left\vert x\right\vert ^{2-d} & \text{for }
& d>2\\
K(\left\vert x\right\vert )=-k_{2}\ln\left\vert x\right\vert  & \text{for} &
d=2
\end{array}
\right.
\]
where $k_{d}=$ cost $>0$, gaussian kernel $K(\left\vert x\right\vert
)=C\exp(-\left\vert x\right\vert ^{2}/\lambda)$ and mollifiers (cf. Section
\ref{convoluzione} in the Appendix).
\end{remark}

\begin{remark}
Examples of functions $g$ which satisfy the conditions (\ref{G2}%
)-(\ref{ipo:g5}) are given by both classical reactions terms as $g(u)=\pm
(u^{3}-u)$ and terms used in recent applications of the Cahn-Hilliard
equations as $g(x,t,u)=\alpha(x,t)u(1-u)$ (\cite{KS}), $g(x,t,u)=\lambda
(x)(h(x)-u)$ (\cite{BEG}) or $g(x,t,u)=-\sigma(x,t)u$ (\cite{BO}) where
$\lambda,h,\alpha$ and $\sigma$ are continuous and positive functions, $h<1$.
\end{remark}

\subsection{Main results\label{ris princ}}

Before stating the main results of this work, let us introduce the definition
of weak solution to system (\ref{eq:CH})-(\ref{eq:ic}).

\begin{definition}
\label{def1}Let $u_{0},K,g$ be such that conditions (\ref{u00})-(\ref{u02}),
(\ref{K1})-(\ref{K4}), (\ref{G2})-(\ref{ipo:g5}) are satisfied. Then, given
$T\in(0,+\infty)$, $u$ is a weak solution to (\ref{eq:CH})-(\ref{eq:ic}) on
$[0,T]$ if
\begin{equation}
u\in H^{1}(0,T,H^{1}(\Omega),\left(  H^{1}(\Omega)\right)  ^{\ast})\text{,}
\label{rego}%
\end{equation}%
\begin{equation}
0\leq u\leq1\quad \text{a.e.~in }Q\text{,} \label{xxx}%
\end{equation}%
\[
w=K\ast(1-2u)\quad \text{a.e.~in }Q\text{,}%
\]%
\[
w\in C([0,T],W^{1,\infty}(\Omega))\text{,}%
\]%
\[
u(0)=u_{0}\text{ in }L^{2}(\Omega)\text{,}%
\]
and the following variational formulation is satisfied almost everywhere  in $(0,T)$ and for every $\psi\in H^{1}(\Omega)$%
\begin{equation}
\left\langle u_{t},\psi\right\rangle +(\mu(u)\nabla w,\nabla\psi)+(\nabla
u,\nabla\psi)=(g(u),\psi)\text{.} \label{CHWEAK}%
\end{equation}

\end{definition}

\begin{remark}
As consequence of (\ref{rego}), $u\in C([0,T],L^{2}(\Omega))$. Hence, the
initial condition (\ref{eq:ic}) makes sense. Moreover, let us note that this notion of solution 
turns out to be particularly useful since it does not involve the potential $f$ and so it can be stated 
for solutions  $u\in [0,1]$, not necessarily different from $0$ and $1$ (cf. also \cite{FGR} for further comments on this point). 
\end{remark}

Here we state our first result whose proof is given in Section~\ref{dimostra}. 
\begin{theorem}
\label{Teorema principale} Let (\ref{K1})-(\ref{K4}), (\ref{u00})-(\ref{u02})
and (\ref{G2})-(\ref{ipo:g5}) be satisfied. Then there exists unique%
\[
u\in H^{1}(0,T,H^{1}(\Omega),(H^{1}(\Omega))^{\ast})(\hookrightarrow
C([0,T],L^{2}(\Omega)))
\]
weak solution to (\ref{eq:CH}) in the sense of Definition \ref{def1}.

Furthermore, if $u_{i}$ $i\in\{1,2\}$, are two solutions to (\ref{eq:CH})-(\ref{eq:bc}) in the
sense of Definition \ref{def1} with initial data $u_{0i}$, $i\in\{1,2\}$,
then, for every $t\in\lbrack0,T]$, the following continuous dependence estimate:
\begin{equation}
\left\Vert u_{1}-u_{2}\right\Vert _{L^{\infty}(0,t,L^{2}(\Omega))}\leq
\exp(Ct)\left\Vert u_{01}-u_{02}\right\Vert _{L^{2}(\Omega)}
\label{dipContLInfty}%
\end{equation}
holds true, where $C>0$ does not depend on $t$ nor on $u_{01}$ and $u_{02}$.
\end{theorem}

The proof is given in Section \ref{dimostra}.

Under additional assumptions on the initial data $u_{0}$ and the function $g$
we obtain more regularity on $u$, as stated in the following result proved in Section~\ref{RRRR}.

\begin{theorem}
\label{teorego} Let the assumptions of Theorem
\ref{Teorema principale} be satisfied. Let $u$ be the weak solution to (\ref{eq:CH})-(\ref{eq:ic}) in
the sense of Definition \ref{def1}. Moreover, assume that $g$ and $u_{0}$ satisfy:
\begin{align}\label{ipo:gReg3}\tag{G4}
&\exists L  >0\text{ such that }\left\vert g(x,t_{1},s)-g(x,t_{1}%
,s)\right\vert \leq L\left\vert t_{1}-t_{2}\right\vert , \quad
\forall t_{1},t_{2}  \in\lbrack0,T]\text{, }\forall x\in\Omega,\,\forall s\in\lbrack
0,1]\text{,}
\end{align}%
\begin{equation}
u_{0}\in H^{2}(\Omega)\text{,} \label{u0h2}%
\end{equation}
and%
\begin{equation}
n\cdot\left(  \nabla(u_{0})+\mu(u_{0})\nabla K\ast(1-2u_{0})\right)  =0\text{
on }\partial\Omega. \label{u0h3}%
\end{equation}
Then $u\in L^{\infty}(0,T,H^{2}(\Omega))$.
\end{theorem}

\begin{remark}
\label{superRemark}Since $u\in L^{\infty}(0,T,H^{2}(\Omega))\cap C(\left[
0,T\right]  ,L^{2}(\Omega))$, thanks to Lemma~\ref{final lemma} in the Appendix, we have $u\in
C(\left[  0,T\right]  ,H^{s}(\Omega))$ for every $s<2$ and hence $u\in
C(\left[  0,T\right]  ,L^{\infty}(\Omega))$ if $d\leq3$.
\end{remark}

If the initial data do not satisfy (\ref{u0h2})-(\ref{u0h3}) the solution
$u$ is more regular only on the set $[T_{0},T]$ for any $T_{0}>0$.

\begin{corollary}
\label{Cororego}Let $u$ be solution to (\ref{eq:CH})-(\ref{eq:ic}) in the
sense of Definition \ref{def1}. Let the assumptions of Theorem
\ref{Teorema principale} be satisfied. Assume that $g$ satisfies
(\ref{ipo:gReg3}). Then for every $T_{0}\in(0,T)$ $~u\in L^{\infty}%
(T_{0},T,H^{2}(\Omega))$.
\end{corollary}

More regularity on $v$ can be obtained under an additional assumption on the
initial datum.

\begin{theorem}
\label{teo v reg}Let the assumption of Theorem \ref{Teorema principale} be
satisfied and let $u_{0}$ such that
\begin{equation}
f^{\prime}(u_{0})\in L^{2}(\Omega). \label{u0h4}%
\end{equation}
Then the weak solution $u$ given by Theorem \ref{Teorema principale} fulfills
\begin{equation}
v\in L^{\infty}(0,T,L^{2}(\Omega))\text{ \ \ \ \ }\nabla v\in L^{2}%
(0,T,L^{2}(\Omega))\text{.} \label{regv}%
\end{equation}

\end{theorem}

\begin{remark}
\label{u not zero}As a consequence of Theorem \ref{teo v reg} the function
$v=f^{\prime}(u)+w$ is well defined. Hence $u\neq0$ and $u\neq1$ $a.e.$ in
$\Omega\times\lbrack0,T]$. Furthermore $u$ also satisfies the weak formulation
given by Definition \ref{def1} with
\[
\left\langle u_{t},\psi\right\rangle +(\mu(u)\nabla v,\nabla\psi
)=(g(u),\psi)\text{, \ \ }v=f^{\prime}(u)+w\text{,}%
\]
instead of (\ref{CHWEAK}).
\end{remark}

Corollary~\ref{Cororego} and Theorem~\ref{teo v reg}
are proved in Section~\ref{RRRR}.

In \cite[Theorem~2.1]{LP2} Londen and Petzeltov\`{a} obtained the separation
properties for the solution to (\ref{eq:CH})-(\ref{eq:ic}) with $g=0$. We
prove here the same results in the case $g$ satisfies (\ref{G2})-(\ref{ipo:g5}).

\begin{theorem}
\label{teo sep}Let the assumptions of Theorem \ref{teorego} be satisfied and
$d\leq3$. Then%
\begin{align}
&\forall T_{0}  \in(0,T)~\exists\, k>0\text{ such that }k\leq u(x,t)\leq
1-k\text{ }\label{SEP}
\text{for }a.a.~x  \in\Omega,t\in( T_{0},T)\text{.}
\end{align}
Furthermore, if%
\begin{equation}
\exists\,\tilde{k}>0\text{ such that }\tilde{k}\leq u_{0}\leq1-\tilde{k}%
\text{,}\label{SEP2}%
\end{equation}
then $T_{0}=0$.
\end{theorem}

\begin{remark}
If $u_{0}$ do not satisfy (\ref{u0h2}) or (\ref{u0h3}), using Corollary
\ref{Cororego} and applying Theorem \ref{teo sep} on the set $(t,T)$ where
$t>0$ is small enough, we can anyway obtain (\ref{SEP}).
\end{remark}

Theorem \ref{teo sep} is proved in Section \ref{Separazione}.

\section{Existence and uniqueness\label{dimostra}}

This section is devoted to the proof of Theorem~\ref{Teorema principale}. We first prove uniqueness of solutions by demonstrating 
estimate \eqref{dipContLInfty}, then we prove existence of solutions by approximating our problem  with a more regular problem $P_\varepsilon$ 
and then passing to the limit as $\varepsilon\to 0$ via suitable a-priori estimates and compactness results. 

\section{Uniqueness}

We now prove the uniqueness of the solution.
In the following formulas the symbol $C$ denotes a positive constant depending
on $T$, $K$, and $g$. It may vary even within the same line.

\begin{proof}
[Proof of (\ref{dipContLInfty})]Let $u_{i}$ and $u_{0i}$ be as in Theorem
\ref{Teorema principale}. Then
\begin{equation}
\left\langle \partial_{t}u_{i},\psi\right\rangle =-\left(  \nabla u_{i}%
+\mu_{i}\nabla w_{i},\nabla\psi\right)  +\left(  g(u_{i}),\psi\right)  \text{
}\forall\psi\in H^{1}(Q), \quad \hbox{a.e.~in }(0,T),  \label{a}%
\end{equation}
where $\mu_{i}=\mu(u_{i})=u_{i}(1-u_{i})$ and $w_{i}=K\ast(1-2u_{i})$.
Computing the difference of (\ref{a}) with $i=1$ and $i=$ $2,$ choosing
$\psi=u:=u_{1}-u_{2}$ and integrating on $(0,t)$, $t\in (0,T]$, we obtain%
\begin{align}
\frac{1}{2}\left\Vert u(t)\right\Vert _{L^{2}(\Omega)}^{2}-\frac{1}%
{2}\left\Vert u_{01}-u_{02}\right\Vert _{L^{2}(\Omega)}^{2}  &  =\int_{0}%
^{t}\left\langle \partial_{t}u,u\right\rangle \label{bb}\\
&  =-\int_{0}^{t}\int_{\Omega}\left\vert \nabla u\right\vert ^{2}\nonumber\\
&  -\int_{0}^{t}\int_{\Omega}(\mu_{1}\nabla w_{1}-\mu_{2}\nabla w_{2})\nabla
u\nonumber\\
&  +\int_{0}^{t}\int_{\Omega}(g\left(  u_{1}\right)  -g\left(  u_{2}\right)
)u\text{.}\nonumber
\end{align}
Using the bounds on $u_{1},u_{2},\mu_{1}$ and $\mu_{2}$ (see (\ref{def:mu})
and (\ref{xxx})) and assumption (\ref{K3}) we obtain the following estimates
\[
\left\vert \int_{\Omega}(\mu_{1}\nabla w_{1}-\mu_{2}\nabla w_{2})\nabla
u\right\vert \leq\frac{1}{2}\int_{\Omega}\left\vert \nabla u\right\vert
^{2}+\frac{1}{2}\int_{\Omega}\left\vert \mu_{1}\nabla w_{1}-\mu_{2}\nabla
w_{2}\right\vert ^{2}%
\]
and
\begin{align*}
\int_{\Omega}\left\vert \mu_{1}\nabla w_{1}-\mu_{2}\nabla w_{2}\right\vert
^{2}   \leq&\int_{\Omega}\left\vert \mu_{1}(\nabla w_{1}-\nabla
w_{2})\right\vert ^{2}+\int_{\Omega}\left\vert (\mu_{1}-\mu_{2})\nabla
w_{2}\right\vert ^{2}\\
\leq &\,C\left\Vert \nabla w_{1}-\nabla w_{2}\right\Vert _{L^{2}(\Omega)}%
^{2}\\
&  +\int_{\Omega}\left\vert (u_{1}-u_{2})(1-u_{1}-u_{2})\nabla w_{2}%
\right\vert ^{2}\\
 \leq &\,C\left\Vert \nabla w_{1}-\nabla w_{2}\right\Vert _{L^{2}(\Omega)}%
^{2}+C\left\Vert \nabla w_{2}\right\Vert _{L^{\infty}(\Omega)}^{2}\left\Vert
u\right\Vert _{L^{2}(\Omega)}^{2}\\
 \leq&\, Cr_{2}^{2}\left\Vert u\right\Vert _{L^{2}(\Omega)}^{2}+Cr_{\infty}%
^{2}\left\Vert u\right\Vert _{L^{2}(\Omega)}^{2}\leq C\left\Vert u\right\Vert
_{L^{2}(\Omega)}^{2}%
\end{align*}
where $r_{2}$ and $r_{\infty}$ as in (\ref{K3}). Furthermore, using
(\ref{ipo:gLip}) we have
\[
\int_{\Omega}(g(u_{1})-g(u_{2}))u\leq\int_{\Omega}Lu^{2}\leq L\left\Vert
u\right\Vert _{L^{2}(\Omega)}^{2}\text{,}%
\]
where $L$ as in (\ref{ipo:gLip}). So, thanks to (\ref{bb}),\ for every
$t\in(0,T)$, we obtain%
\[
\left\Vert u(t)\right\Vert _{L^{2}(\Omega)}^{2}\leq2\left\Vert u_{01}%
-u_{02}\right\Vert _{L^{2}(\Omega)}^{2}+C\int_{0}^{t}\left\Vert u\right\Vert
_{L^{2}(\Omega)}^{2}\text{.}%
\]
Using the Gronwall's Lemma, we get (\ref{dipContLInfty}), and so also uniqueness of solutions is proved.
\end{proof}

\section{Existence}

In order to show the existence of the solution to (\ref{eq:CH})--(\ref{eq:ic})
we study an approximate problem $P_{\varepsilon}$ depending on a parameter
$\varepsilon$. We prove the existence of the solution $u_{\varepsilon}$ to
$P_{\varepsilon}$ and, finally, we obtain $u$ as limit (for $\varepsilon
\rightarrow0$) of $u_{\varepsilon}$ in a proper functional space.

\subsection{Approximate problem $P_{\varepsilon}$\label{commenti}}

We start extending the domain of the function $g(x,t,s)$ to every $s\in%
\mathbb{R}
$ since we cannot prove that the solution $u_{\varepsilon}$ to the approximate
problem satisfies the condition $u_{\varepsilon}\in\lbrack0,1]$ for
$a.a.~x\in\Omega,t\in\lbrack0,T]$. Let us define the function $g^{1}%
:\Omega\times%
\mathbb{R}
^{+}\times%
\mathbb{R}
\rightarrow%
\mathbb{R}
$:%
\[
\left\{
\begin{array}
[c]{lll}%
g^{1}(x,t,s)=g(x,t,0) & \forall x\in\Omega & \forall t\in%
\mathbb{R}
^{+},~s\leq0\\
g^{1}(x,t,s)=g(x,t,s) & \forall x\in\Omega & \forall t\in%
\mathbb{R}
^{+},~s\in\lbrack0,1]\\
g^{1}(x,t,s)=g(x,t,1) & \forall x\in\Omega & \forall t\in%
\mathbb{R}
^{+},~s\geq1
\end{array}
\right.  \text{.}%
\]

We remark that $g^{1}$ satisfies (\ref{G2})-(\ref{ipo:g5}). Furthermore%

\begin{equation}
\left\vert g^{1}(x,t,s)\right\vert \leq C\quad\forall s\in%
\mathbb{R}
~\forall(x,t)\in Q \label{ipo:g3bis}%
\end{equation}
where $C$ as in (\ref{ipo:Gi3}) and
\begin{align}
g^{1}(x,t,s_{1})   \geq0\geq g^{1}(x,t,s_{2})~\label{G5bis}\quad
\forall t   \in%
\mathbb{R}
^{+},\,\forall x\in\Omega,\,\forall s_{1}\leq0,\, \forall s_{2}\geq1\text{.}
\end{align}

Let us consider the approximate problem $P_{\varepsilon}$: find a solution $u$ (we do not use the symbol $u_\varepsilon$ for simplicity of notation)
to%
\begin{equation}
\left\langle \partial_t u,\psi\right\rangle +\left(
\mu_{\varepsilon}\nabla v,\nabla\psi\right)  =\left(  g^{1}(u),\psi\right)
\text{ }\forall\psi\in H^{1}(\Omega), \quad \hbox{a.e.~in }(0,T)\text{,} \label{eq:CHrel}%
\end{equation}%
\begin{equation}
v=f_{\varepsilon}^{\prime}(u)+w\quad \hbox{a.e.~in }Q \label{eq:vrel}%
\end{equation}%
\begin{equation}
w=K\ast(1-2u)\quad \text{a.e.~in }Q\text{,} \label{eq:wrel}%
\end{equation}%
\begin{equation}
n\cdot\mu_{\varepsilon}\nabla v=0\quad \text{a.e.~on }\Gamma\text{,}
\label{eq:bcrel}%
\end{equation}%
\begin{equation}
u(0,x)=u_{0}(x), \quad \text{for a.a.}~x\in\Omega\text{,} \label{eq:icrel}%
\end{equation}
where
\begin{equation}
\mu_{\varepsilon}=\max\{\mu+\varepsilon,\varepsilon\} \label{def:mu-epsilon}%
\end{equation}
and $f_{\varepsilon}$ is the solution to the following Cauchy-problem:%
\begin{equation}
f_{\varepsilon}^{\prime\prime}=\left(  1+2a_{\varepsilon}\right)  \frac{1}%
{\mu_{\varepsilon}}\text{ ,~}f_{\varepsilon}^{\prime}(\frac{1}{2})=f^{\prime
}(\frac{1}{2}),\text{ and }f_{\varepsilon}(\frac{1}{2})=f(\frac{1}{2})\text{,}
\label{def:f-epsiolon}%
\end{equation}
where $a_{\varepsilon}=\frac{\left(  1+4\varepsilon\right)  ^{1/2}-1}{2}$.
Thanks to (\ref{def:mu}) and (\ref{def:mu-epsilon}), we have
\begin{equation}
\mu_{\varepsilon}(s)=\left\{
\begin{array}
[c]{ll}%
\varepsilon & \text{for }s<0\\
\left(  s+a_{\varepsilon}\right)  \left(  1+a_{\varepsilon}-s\right)  &
\text{for }s\in\lbrack0,1]\\
\varepsilon & \text{for }s>0
\end{array}
\right.  \text{.} \label{yui}%
\end{equation}
Hence, $\mu_{\varepsilon}$ is continuous. We remark that $\mu_{\varepsilon
}(s)$ is not decreasing for $s\leq1/2$ and not increasing for $s\geq1/2$. This
yields
\begin{equation}
\varepsilon\leq\mu_{\varepsilon}\leq\mu_{\varepsilon}(1/2)=\frac
{1+4\varepsilon}{4}\text{.} \label{mu-eps-lim}%
\end{equation}
From (\ref{def:f-epsiolon}) and (\ref{yui}) it follows%
\begin{equation}
f_{\varepsilon}^{\prime\prime}(s)=\left\{
\begin{array}
[c]{ll}%
\frac{1+2a_{\varepsilon}}{\varepsilon} & \text{for }s<0\\
\frac{1+2a_{\varepsilon}}{\left(  s+a_{\varepsilon}\right)  \left(
1+a_{\varepsilon}-s\right)  } & \text{for }s\in\lbrack0,1]\\
\frac{1+2a_{\varepsilon}}{\varepsilon} & \text{for }s>0
\end{array}
\right.  \label{we giovane}%
\end{equation}
and, in particular,
\begin{equation}
0<f_{\varepsilon}^{\prime\prime}(s)\leq\frac{1+2a_{\varepsilon}}{\varepsilon
}\text{.} \label{erba}%
\end{equation}
Furthermore $f_{\varepsilon}^{\prime\prime}$ satisfies the symmetry property%
\begin{equation}
f_{\varepsilon}^{\prime\prime}\left(  \frac{1}{2}+s\right)  =f_{\varepsilon
}^{\prime\prime}\left(  \frac{1}{2}-s\right)  ~\forall s\in%
\mathbb{R}
\text{.} \label{sim1}%
\end{equation}
Thanks to (\ref{erba}), $f_{\varepsilon}^{\prime}$ is increasing and, thanks
to $f_{\varepsilon}^{\prime}(1/2)=f(1/2)=0$, $f_{\varepsilon}^{\prime}(s)<0$
for $s<1/2$ and $f_{\varepsilon}^{\prime}(s)>0$ for $s>1/2$. Using
(\ref{we giovane}) we now obtain
\begin{equation}
\left\{
\begin{array}
[c]{lll}%
f_{\varepsilon}^{\prime}(s)<0 & \text{for} & s<0\\
f_{\varepsilon}^{\prime}(s)=\ln\left(  \frac{a_{\varepsilon}+s}%
{1+a_{\varepsilon}-s}\right)  & \text{for} & s\in\lbrack0,1]\\
f_{\varepsilon}^{\prime}(s)>0 & \text{for} & s>1\text{.}%
\end{array}
\right.  \label{propf}%
\end{equation}
Furthermore $f_{\varepsilon}^{\prime}$ satisfies%
\begin{equation}
f_{\varepsilon}^{\prime}\left(  \frac{1}{2}+s\right)  =-f_{\varepsilon
}^{\prime}\left(  \frac{1}{2}-s\right)  ~\forall s\in%
\mathbb{R}
\text{.} \label{sim3}%
\end{equation}
Since $f_{\varepsilon}^{\prime\prime}\leq\frac{1+2a_{\varepsilon}}%
{\varepsilon}$ and $f_{\varepsilon}^{\prime}(1/2)=0$, we have $f_{\varepsilon
}^{\prime}(s)\leq\frac{1+2a_{\varepsilon}}{\varepsilon}(s-1/2)$ for $s\geq
1/2$. So, using (\ref{sim3}), we get
\begin{equation}
\left\vert f_{\varepsilon}^{\prime}(s)\right\vert \leq\frac{1+2a_{\varepsilon
}}{\varepsilon}\left\vert s-1/2\right\vert ~\forall s\in%
\mathbb{R}
\text{.} \label{qwe}%
\end{equation}
As a consequence of (\ref{propf}) $s=\frac{1}{2}$ minimizes $f_{\varepsilon
}(s)$. From (\ref{sim3}) we have
\begin{equation}
f_{\varepsilon}\left(  \frac{1}{2}+s\right)  =f_{\varepsilon}\left(  \frac
{1}{2}-s\right)  ~\forall s\in%
\mathbb{R}
\text{.} \label{sim4}%
\end{equation}
Now, we show that
\begin{equation}
f_{\varepsilon}(s)\geq\frac{1}{2\varepsilon}s^{2}-c_{\varepsilon}~\forall s\in%
\mathbb{R}
, \label{propf2}%
\end{equation}
where $c_{\varepsilon}$ is a positive constant depending on $\varepsilon$. We
start showing
\begin{equation}
f_{\varepsilon}(s)\geq\frac{1+a_{\varepsilon}}{2\varepsilon}\left(
s-1/2\right)  ^{2}-c_{\varepsilon}^{\prime}~\forall s\in%
\mathbb{R}
\label{tgb}%
\end{equation}
where $c_{\varepsilon}^{\prime}$ is a positive constant depending on
$\varepsilon$. We prove (\ref{tgb}) for $s>1/2$; the proof for $s<1/2$ can be
obtained using (\ref{sim4}). As a consequence of (\ref{we giovane}) we have
$f_{\varepsilon}^{\prime}(s)=\frac{1+2a_{\varepsilon}}{\varepsilon
}(s-1)+f_{\varepsilon}^{\prime}(1)$, $s>1$. Furthermore $f_{\varepsilon
}^{\prime}(s)\geq0$ for $s>1/2$ as\ a consequence of (\ref{propf}). Hence
$f_{\varepsilon}^{\prime}(s)\geq\frac{1+2a_{\varepsilon}}{\varepsilon}%
s-\frac{1+2a_{\varepsilon}}{\varepsilon}~\forall s>1/2$ (the right term is
negative for $s\in\lbrack1/2,1]$). From the last inequality follows by
integration 
\begin{align*}
f_{\varepsilon}(s)-f_{\varepsilon}(1/2)  &  \geq\frac{1+2a_{\varepsilon}%
}{2\varepsilon}s^{2}-\frac{1+2a_{\varepsilon}}{\varepsilon}s-\frac
{1+2a_{\varepsilon}}{2\varepsilon}\frac{1}{4}+\frac{1+2a_{\varepsilon}%
}{\varepsilon}\frac{1}{2}\\
&  \geq\frac{1+2a_{\varepsilon}}{2\varepsilon}s^{2}-\delta s^{2}%
-\frac{1+2a_{\varepsilon}}{2\varepsilon}\frac{1}{4\delta}-\frac
{1+2a_{\varepsilon}}{2\varepsilon}\frac{1}{4}+\frac{1+2a_{\varepsilon}%
}{\varepsilon}\quad \forall\delta>0\text{.}%
\end{align*}
We take into account $\frac{1+2a_{\varepsilon}}{2\varepsilon}>\frac
{1+a_{\varepsilon}}{2\varepsilon}$, choose $\delta$ suitably and get
(\ref{tgb}). Hence,
\begin{align*}
\frac{1+a_{\varepsilon}}{2\varepsilon}\left(  s-1/2\right)  ^{2}  &
=\frac{1+a_{\varepsilon}}{2\varepsilon}(s^{2}-s-1/4)\\
&  \geq\frac{1+a_{\varepsilon}}{2\varepsilon}((1-\delta)s^{2}-1/4-\frac
{1}{8\delta})\quad \forall\delta>0\,.
\end{align*}
Choosing $\delta$ suitably small and using $\frac{1+a_{\varepsilon}}{2\varepsilon
}>\frac{1}{2\varepsilon}$ we have (\ref{propf2}).

\subsection{Existence of the solution to the approximate
problem\label{esist prob appr}}

The following lemma states the existence of a solution to (\ref{eq:CHrel}%
)-(\ref{eq:icrel}) for a fixed $\varepsilon>0$ small enough.

\begin{lemma}
\label{probdisc}Let \ $\varepsilon<\frac{1}{2r_{2}}$ ($r_{2}$ as in
(\ref{K3})). Let (\ref{K1})-(\ref{K3}), (\ref{ipo:gLip}), (\ref{G2}) and
(\ref{ipo:g3bis}) be satisfied. Then there exists
\[
u\in H^{1}(0,T,H^{1}(\Omega),(H^{1}(\Omega))^{\ast})\cap L^{\infty}%
(0,T,L^{2}(\Omega))
\]
solution to (\ref{eq:CHrel})-(\ref{eq:icrel}) such that
\[
\left\Vert \mu_{\varepsilon}^{1/2}(u)\left\vert \nabla v\right\vert
\right\Vert _{L^{2}(0,T,L^{2}(\Omega))}\leq C
\]
where $C$ is a positive constant depending on $\varepsilon$.
\end{lemma}

\begin{proof}
The argument is based on a Faedo-Galerkin's approximation scheme. We introduce
the family $\{e_{i}\}_{i\in%
\mathbb{N}
}$ of eigenfunctions of $-\Delta+Id:V\rightarrow V^{\ast}$ as a Galerkin base
in $V=H^{1}(\Omega)$. We define the orthogonal projector $P_{n}:H=L^{2}%
(\Omega)\rightarrow V_{n}=span(\{e_{i}\}_{i=1}^{n})$ and $u_{0n}=P_{n}u_{0}$.
We then look for functions of the form \
\[
u_{n}(t)=\sum_{k=1}^{n}\alpha_{k}(t)e_{k}\text{ and }v_{n}(t)=\sum_{k=1}%
^{n}\beta_{k}(t)e_{k}%
\]
which solve the following approximating problem%
\begin{align}
(u_{n}^{\prime},\psi)+(\mu_{\varepsilon}(u_{n})\nabla v_{n},\nabla\psi)  &
=(g_{n}^{1},\psi)\text{ }\forall\psi\in V_{n}\label{eq:CHRelDisc}\\
\nonumber
v_{n}  &  =P_{n}(K\ast(1-2u_{n})+f_{\varepsilon}^{\prime}(u_{n}))\\
\nonumber
g_{n}^{1}  &  =P_{n}\left(  g^{1}(u_{n})\right) \\
u_{n}(0)  &  =u_{0n}\text{.} \label{yu}%
\end{align}
This approximating problem is equivalent to solve a Cauchy problem for a
system of ODEs in the $n$ unknowns ($\alpha_{i}$). As a consequence of
(\ref{def:mu-epsilon}), (\ref{G2}), (\ref{ipo:gLip}) and (\ref{def:f-epsiolon}%
), for every $\psi\in V_{n}$, the functions $(m(u_{n})\nabla v_{n},\nabla
\psi)$ and $(g_{n},\psi)$ are locally Lipschitz with respect to the variables
$\alpha_{i}$ uniformly in $t$. Hence there exists $T_{n}\in%
\mathbb{R}
_{+}$ such that system (\ref{eq:CHRelDisc}) has an unique solution $\alpha
_{1},\dots,\alpha_{n},\beta_{1},\dots,\beta_{n}\in C^{1}([0,T_{n});%
\mathbb{R}
)$.

We now want to prove a-priori estimates for $u_{n}$ uniformly in $n$. Henceforth we shall denote
by $C$ a positive constant which depend on $\varepsilon$, but it is
independent of $n$ and $t$. The values of $C$ may possibly vary even within
the same line. We
choose $\psi=v_{n}$ as test function and get%
\[
(u_{n}^{\prime},v_{n})+(\mu_{\varepsilon}(u_{n})\nabla v_{n},\nabla
v_{n})=(g_{n}^{1},v_{n})\text{.}%
\]
Thus,%
\begin{align*}
(u_{n}^{\prime},v_{n})  &  =(u_{n}^{\prime},f_{\varepsilon}^{\prime}%
(u_{n}))+(u_{n}^{\prime},K\ast(1-2u_{n}))\\
&  =\frac{d}{dt}\left(  \int_{\Omega}f_{\varepsilon}(u_{n})+\int_{\Omega}%
\int_{\Omega}K(x-y)u_{n}(x)(1-u_{n}(y))\right)  \text{.}%
\end{align*}
From this follows by integration on $(0,t)$:%
\begin{align}
&  \left(  \int_{\Omega}f_{\varepsilon}(u_{n})+\int_{\Omega}\int_{\Omega
}K(x-y)u_{n}(x)(1-u_{n}(y))\right)  (t)\label{Test1}+\int_{0}^{t}\int_{\Omega}\mu_{\varepsilon}(u_{n})\left\vert \nabla
v_{n}\right\vert ^{2}\\
&  =\int_{0}^{t}(g_{n}^{1},v_{n})+\left(  \int_{\Omega}f_{\varepsilon}%
(u_{n})+\int_{\Omega}\int_{\Omega}K(x-y)u_{n}(x)(1-u_{n}(y))\right)
(0)\text{.}\nonumber
\end{align}
Thanks to (\ref{qwe}) we have $\left\vert f_{\varepsilon}^{\prime
}(s)\right\vert \leq C\left\vert s\right\vert +C$. Due to (\ref{ipo:g3bis}), we have%
\begin{equation}
(g_{n}^{1},f_{\varepsilon}^{\prime}(u_{n}))\leq C+C\left\Vert u_{n}\right\Vert
_{H}^{2}\text{.} \label{tuc}%
\end{equation}
Using (\ref{propf2}) and (\ref{K3}), we obtain, for $\delta>0$ to be
announced,%
\begin{align}
&  \int_{\Omega}\int_{\Omega}K(x-y)u_{n}(x)(1-u_{n}(y))\label{ert}\\
+\int_{\Omega}f_{\varepsilon}(u_{n})  &  \geq\frac{1}{2\varepsilon}%
\int_{\Omega}u_{n}{}^{2}-c_{\varepsilon}+(K\ast(1-u_{n}),u_{n})_{H}\nonumber\\
&  \geq\frac{1}{2\varepsilon}\left\Vert u_{n}\right\Vert _{H}^{2}%
-c_{\varepsilon}-r_{2}\left\Vert u_{n}\right\Vert _{H}\left\Vert
1-u_{n}\right\Vert _{H}\nonumber\\
&  \geq\left(  \frac{1}{2\varepsilon}-r_{2}\right)  \left\Vert u_{n}%
\right\Vert _{H}^{2}-C_{\varepsilon}-r_{2}\left\vert \Omega\right\vert
\left\Vert u_{n}\right\Vert _{H}\nonumber\\
&  \geq\left(  \frac{1}{2\varepsilon}-r_{2}-\delta\right)  \left\Vert
u_{n}\right\Vert _{H}^{2}-C_{\delta,\varepsilon}\text{,}\nonumber
\end{align}
where $C_{\delta,\varepsilon}$ denotes a constant depending on both
$\varepsilon$ and $\delta$. Since $\frac{1}{2\varepsilon}>r_{2}$, we choose
$\delta$ such that $\left(  \frac{1}{2\varepsilon}-r_{2}-\delta\right)
=C>0$.\ From (\ref{ipo:g3bis}) and (\ref{K3}) follows%
\begin{align}
(g_{n}^{1},K\ast(1-2u_{n}))  &  \leq C\left\Vert K\ast(1-2u_{n})\right\Vert
_{H}\label{toc}\\
&  \leq C+D\left\Vert u_{n}\right\Vert _{H}\leq C+D\left\Vert u_{n}\right\Vert
_{H}^{2}\text{.}\nonumber
\end{align}
Using (\ref{Test1}), (\ref{tuc}), (\ref{ert}) and (\ref{toc}) we get%
\begin{equation}
\left\Vert u_{n}(t)\right\Vert _{H}^{2}+\int_{0}^{t}\int_{\Omega}%
\mu_{\varepsilon}(u_{n})\left\vert \nabla v_{n}\right\vert ^{2}\leq
C+D\int_{0}^{t}\left\Vert u_{n}\right\Vert _{H}^{2}\text{.} \label{test2}%
\end{equation}
We now use Gronwall's Lemma and get the estimates
\begin{equation}
\left\Vert u_{n}\right\Vert _{L^{\infty}(0,T,H)}\leq C \label{uru}%
\end{equation}
and%
\begin{equation}
\left\Vert \mu_{\varepsilon}^{1/2}(u_{n})\left\vert \nabla v_{n}\right\vert
\right\Vert _{L^{2}(0,T,H)}\leq C\text{.} \label{ss}%
\end{equation}
Furthermore, as consequence of (\ref{mu-eps-lim}) and (\ref{K3}), we obtain,
for every $\delta>0$ and some $C_{\delta}>0$,
\begin{align*}
\int_{\Omega}\mu_{\varepsilon}(u_{n})\left\vert \nabla v_{n}\right\vert ^{2}
&  \geq\varepsilon\int_{\Omega}\left\vert \nabla v_{n}\right\vert ^{2}\\
&  =\varepsilon\int_{\Omega}\left\vert \frac{\left(  1+2a_{\varepsilon
}\right)  \nabla u_{n}}{\mu_{\varepsilon}(u_{n})}+\nabla K\ast(1-2u_{n}%
)\right\vert ^{2}\\
&  \geq c\int_{\Omega}\left\vert \nabla u_{n}\right\vert ^{2}+C\int_{\Omega
}\left\vert \nabla K\ast(1-2u_{n})\right\vert ^{2}\\
&  \quad +C\int_{\Omega}\nabla u_{n}\nabla K\ast(1-2u_{n})\\
&  \geq\left(  c-\delta\right)  \int_{\Omega}\left\vert \nabla u_{n}%
\right\vert ^{2}-\left(  C+C_{\delta}\right)  r_{2}\left(  \left\Vert
u_{n}\right\Vert _{L^{2}(\Omega)}^{2}+1\right)
\end{align*}
where $r_{2}$ as in (\ref{K3}) and $c$ is a positive constant depending on
$\varepsilon$. If $\delta$ is small enough, then%
\[
\int_{\Omega}\mu_{\varepsilon}(u_{n})\left\vert \nabla v_{n}\right\vert
^{2}\geq c\int_{\Omega}\left\vert \nabla u_{n}\right\vert ^{2}-C\left\Vert
u_{n}\right\Vert _{L^{2}(\Omega)}^{2}-C\text{.}%
\]
Hence, from (\ref{test2}) we get%
\begin{align}
\left\Vert u_{n}\right\Vert _{L^{2}(0,T,V)}  &  \leq C\text{,}%
\label{uuuuuuuuu}\\
\left\Vert \nabla v_{n}\right\Vert _{L^{2}(0,T,H)}  &  \leq C\text{.}
\label{hihi}%
\end{align}
Furthermore, (\ref{K3}), (\ref{qwe}) and (\ref{uru}) yield
\begin{align*}
\left\vert \overline{v_{n}}\right\vert  &  =\left\vert \frac{1}{\left\vert
\Omega\right\vert }\int_{\Omega}v_{n}\right\vert =C\left\vert \int_{\Omega
}f_{\varepsilon}^{\prime}(u_{n})+\int_{\Omega}K\ast(1-2u_{n})\right\vert \\
&  \leq C\left\Vert u_{n}\right\Vert _{H}^{2}+C+\left\Vert K\ast
(1-2u_{n})\right\Vert _{H}^{2} \leq C\left\Vert u_{n}\right\Vert _{H}^{2}+C\leq C\text{.}%
\end{align*}
Using the Poincar\'{e}-Wirtinger inequality we get%
\begin{equation}
\left\Vert v_{n}\right\Vert _{L^{2}(0,T,V)}\leq C\text{.} \label{vn lim}%
\end{equation}
Moreover, thanks to (\ref{ipo:g3bis}), we obtain%
\begin{equation}
\left\Vert g_{n}^{1}\right\Vert _{L^{2}(Q)}\leq C\text{.} \label{g lim}%
\end{equation}
In order to estimate $u_{n}^{\prime}$, from (\ref{eq:CHRelDisc}), using
(\ref{mu-eps-lim}), we obtain
\begin{align*}
\left\langle u_{n}^{\prime},\psi\right\rangle    =-(\mu_{\varepsilon}%
(u_{n})\nabla v_{n},\nabla\psi)+(g_{n},\psi)&\leq C\left\Vert \nabla
v_{n}\right\Vert _{H}\left\Vert \nabla\psi\right\Vert _{H}+\left\Vert
g_{n}^{1}\right\Vert _{H}\left\Vert \psi\right\Vert _{H}\\
&\leq(C\left\Vert \nabla v_{n}\right\Vert _{H}+\left\Vert g_{n}%
^{1}\right\Vert _{H})\left\Vert \psi\right\Vert _{V}\text{.}%
\end{align*}
So, the estimates (\ref{hihi}) and (\ref{g lim}) yield%
\[
\left\Vert u_{n}^{\prime}\right\Vert _{L^{2}(0,T,V^{\ast})}\leq C\text{.}%
\]
Using compactness results, we obtain for a not relabeled subsequence%
\begin{equation}
u_{n}\rightharpoonup u~\text{\ weakly in }L^{2}(0,T,V)\text{,} \label{11}%
\end{equation}%
\begin{equation}
u_{n}\rightharpoonup u~\text{\ weakly* in }L^{\infty}(0,T,H)\text{,}
\label{12}%
\end{equation}%
\begin{equation}
u_{n}^{\prime}\rightharpoonup u^{\prime}~\text{\ weakly in }L^{2}(0,T,V^{\ast
})\text{,} \label{01}%
\end{equation}%
\begin{equation}
f_{\varepsilon}^{\prime}(u_{n})\rightharpoonup f_{\varepsilon}^{\ast
}~\text{\ weakly* in }L^{\infty}(0,T,H)\text{,} \label{de}%
\end{equation}%
\begin{equation}
v_{n}\rightharpoonup v\text{ weakly in }L^{2}(0,T,V)\text{.} \label{da}%
\end{equation}
Taking into account Theorem \ref{Robinson} in the Appendix, we have
\begin{equation}
u_{n}\rightarrow u~\text{\ strongly in }L^{2}(0,T,H)\text{ and }%
a.e.~\text{\ in }Q\text{.} \label{dede}%
\end{equation}
Functions $\mu_{\varepsilon}$ and $g^{1}$ are continuous, so, using
(\ref{ipo:g3bis}) and (\ref{mu-eps-lim}), we have (thanks to dominated
convergence Theorem)
\begin{equation}
\mu_{\varepsilon}(u_{n})\rightarrow\mu_{\varepsilon}(u)\text{ }a.e.\text{ in
}Q\text{\ ,} \label{limite mu-eps}%
\end{equation}%
\begin{equation}
g^{1}(u_{n})\rightarrow g^{1}(u)~\text{\ in }L^{2}(0,T,H).
\end{equation}
Hence%
\begin{equation}
\mu_{\varepsilon}(u_{n})\nabla v_{n}\rightharpoonup\mu_{\varepsilon}(u)\nabla
v~\text{\ weakly in }L^{2}(Q)\text{. } \label{s}%
\end{equation}
Indeed, let $\psi\in L^{2}(0,T,H)$, $i\in\{1,\dots,d\}$. From
\[
\int_{0}^{T}\left(  \mu_{\varepsilon}(u_{n})\partial_{i}v_{n},\psi\right)
=~\int_{0}^{T}\left(  \partial_{i}v_{n},\mu_{\varepsilon}(u_{n})\psi\right)
\]
we get
\begin{align*}
\int_{0}^{T}\left(  \partial_{i}v_{n},\mu_{\varepsilon}(u_{n})\psi\right)   
=\int_{0}^{T}\left(  \partial_{i}v,\mu_{\varepsilon}(u)\psi\right) +\int_{0}^{T}\left(  \partial_{i}v_{n}-\partial_{i}v,\mu_{\varepsilon
}(u)\psi\right)  +\int_{0}^{T}\left(  \partial_{i}v_{n},\left(  \mu
_{\varepsilon}(u_{n})-\mu_{\varepsilon}(u)\right)  \psi\right)  .
\end{align*}
Thanks to (\ref{mu-eps-lim}), (\ref{vn lim}) and (\ref{limite mu-eps}), using
dominated convergence Theorem we obtain
\begin{align*}
 \left\vert \int_{0}^{T}\left(  \partial_{i}v_{n},\left(  \mu_{\varepsilon
}(u_{n})-\mu_{\varepsilon}(u)\right)  \psi\right)  \right\vert & \leq\left\Vert \partial_{i}v_{n}\right\Vert _{L^{2}(0,T,H)}\left\Vert
\left(  \mu_{\varepsilon}(u_{n})-\mu_{\varepsilon}(u)\right)  \psi\right\Vert
_{L^{2}(0,T,H)}\\
&  \leq C\left\Vert \left(  \mu_{\varepsilon}(u_{n})-\mu_{\varepsilon
}(u)\right)  \psi\right\Vert _{L^{2}(0,T,H)}\rightarrow0
\end{align*}
for $n\rightarrow+\infty$. Furthermore, as consequence of (\ref{mu-eps-lim}),
$\mu_{\varepsilon}(u)\psi\in L^{2}(0,T,H)$ and so, thanks to (\ref{da}), we
have $\int_{0}^{T}\left(  \partial_{i}v-\partial_{i}v_{n},\mu_{\varepsilon
}(u)\psi\right)  \rightarrow0$ for $n\rightarrow+\infty$. This yields (\ref{s}).

Finally, using (\ref{de}), (\ref{dede}) and continuity of $f_{\varepsilon
}^{\prime}$, we have $f_{\varepsilon}^{\ast}=f_{\varepsilon}^{\prime}(u).$ The
convergences (\ref{11})-(\ref{da}), (\ref{dede}), (\ref{limite mu-eps}%
)-(\ref{s}) are enough to pass to the limit ($n\rightarrow+\infty$) in
(\ref{eq:CHRelDisc})-(\ref{yu}) and to deduce that $u$ is solution to
(\ref{eq:CHrel}).

Furthermore, thanks to Fatou Lemma and (\ref{ss}), we get%
\[
\left\Vert \mu_{\varepsilon}^{1/2}(u)\left\vert \nabla v\right\vert
\right\Vert _{L^{2}(0,T,L^{2}(\Omega))}\leq\lim\inf_{n\rightarrow\infty
}\left\Vert \mu_{\varepsilon}^{1/2}(u_{n})\left\vert \nabla v_{n}\right\vert
\right\Vert _{L^{2}(0,T,L^{2}(\Omega))}\leq C\text{.}%
\]
Lemma \ref{probdisc} is now proved.
\end{proof}

\subsection{Passing to the limit as $\varepsilon\rightarrow0$}
\label{pass lim}

In order to show Theorem \ref{Teorema principale} it is necessary to pass to
the limit $\varepsilon\rightarrow0$ in (\ref{eq:CHrel})-(\ref{eq:icrel}). Hence, we need to perform here uniform - with respect to $\varepsilon$ - 
estimates on the solution $\left(  u_{\varepsilon},v_{\varepsilon},w_{\varepsilon}\right)  $ to (\ref{eq:CHrel})-(\ref{eq:icrel}). Henceforth
we shall denote by $C$ a positive constant which doesn't depend on
$\varepsilon$ and $t$. The values of $C$ may possibly vary even within the
same line. 

Let us choose now $\psi
=u_{\varepsilon}$ as test function in (\ref{eq:CHrel}). We get (using
(\ref{mu-eps-lim}) and assumptions (\ref{K3}) and (\ref{ipo:g3bis}))
\begin{align*}
\frac{1}{2}\frac{d}{dt}\left\Vert u_{\varepsilon}\right\Vert _{L^{2}(\Omega
)}^{2}  &  =\left\langle u_{\varepsilon}^{\prime},u_{\varepsilon}\right\rangle
=-\int_{\Omega}\mu_{\varepsilon}\nabla u_{\varepsilon}\nabla v_{\varepsilon
}+\int_{\Omega}u_{\varepsilon}g^{1}(u_{\varepsilon})\\
&  \leq-\int_{\Omega}\left\vert \nabla u_{\varepsilon}\right\vert ^{2}%
-\int_{\Omega}\mu_{\varepsilon}\nabla u_{\varepsilon}\nabla w_{\varepsilon
}+C\left\Vert u_{\varepsilon}\right\Vert _{L^{2}(\Omega)}^{2}+C\\
&  \leq-\int_{\Omega}\left\vert \nabla u_{\varepsilon}\right\vert
^{2}+C\left\Vert \nabla u_{\varepsilon}\right\Vert _{L^{2}(\Omega)}\left\Vert
K\ast(1-2u_{\varepsilon})\right\Vert _{H^1(\Omega)}\\
&\quad+C\left\Vert u_{\varepsilon}\right\Vert _{L^{2}(\Omega)}^{2}+C\\
&  \leq(\delta-1)\int_{\Omega}\left\vert \nabla u_{\varepsilon}\right\vert
^{2}+C_{\delta}\left\Vert u_{\varepsilon}\right\Vert _{L^{2}(\Omega)}%
^{2}+C_{\delta}%
\end{align*}
for every $\delta>0$ and some $C_{\delta}$ depending on $\delta$. Moving $(\delta-1)\int_{\Omega}\left\vert \nabla u_{\varepsilon
}\right\vert ^{2}$ on the left side of the inequality, choosing $\delta<1$ and
using Gronwall's Lemma we get
\begin{equation}
\left\Vert u_{\varepsilon}\right\Vert _{L^{\infty}(0,T,L^{2}(\Omega
))}+\left\Vert \nabla u_{\varepsilon}\right\Vert _{L^{2}(0,T,L^{2}(\Omega
))}\leq C \label{mnb}%
\end{equation}
and therefore%
\begin{equation}
\left\Vert u_{\varepsilon}\right\Vert _{L^{2}(0,T,H^{1}(\Omega))}\leq
C\text{.} \label{iop}%
\end{equation}

Using $\psi=v_{\varepsilon}$ as test function  in \eqref{eq:CHrel}, we have%
\begin{align}
&  \frac{d}{dt}\left\{  \int_{\Omega}f_{\varepsilon}(u_{\varepsilon}%
)+\int_{\Omega}\left[  K\ast(1-u_{\varepsilon})\right]  u_{\varepsilon
}\right\}  +\int_{\Omega}\mu_{\varepsilon}\left\vert \nabla v_{\varepsilon
}\right\vert ^{2}\label{ret}\\
&  =\left\langle u_{\varepsilon}^{\prime},v_{\varepsilon}\right\rangle
+\int_{\Omega}\mu_{\varepsilon}\left\vert \nabla v_{\varepsilon}\right\vert
^{2}=\int_{\Omega}g^{1}(u_{\varepsilon})f_{\varepsilon}^{\prime}%
(u_{\varepsilon})+\int_{\Omega}g^{1}(u_{\varepsilon})w_{\varepsilon}%
\text{.}\nonumber
\end{align}
Thanks to (\ref{ipo:g3bis}), (\ref{K3}) and (\ref{mnb}), we infer
\[
\int_{\Omega}g^{1}(u_{\varepsilon})w_{\varepsilon}\leq C\left\Vert
w_{\varepsilon}\right\Vert _{L^{2}(\Omega)}\leq C\left\Vert u_{\varepsilon
}\right\Vert _{L^{2}(\Omega)}\leq C
\]
and%
\[
\left\vert \int_{\Omega}\left[  K\ast(1-u_{\varepsilon})\right]
u_{\varepsilon}\right\vert \leq\left\Vert u_{\varepsilon}\right\Vert
_{L^{2}(\Omega)}\left\Vert K\ast(1-2u_{\varepsilon})\right\Vert _{L^{2}%
(\Omega)}\leq C\text{.}%
\]
Moreover, (\ref{G5bis}) and (\ref{propf}) yield the following estimate
\begin{align}
\int_{\Omega}g^{1}(u_{\varepsilon})f_{\varepsilon}^{\prime}(u_{\varepsilon})
 =&\int_{u_{\varepsilon}\leq0}g^{1}(u_{\varepsilon})f_{\varepsilon}^{\prime
}(u_{\varepsilon})\label{stimaimp}\\
& \, +\int_{u_{\varepsilon}\geq0}g^{1}(u_{\varepsilon})f_{\varepsilon}^{\prime
}(u_{\varepsilon})+\int_{u_{\varepsilon}\in(0,1)}g^{1}(u_{\varepsilon
})f_{\varepsilon}^{\prime}(u_{\varepsilon})\nonumber\\
 \leq&\int_{u_{\varepsilon}\in(0,1)}g^{1}(u_{\varepsilon})\ln(u_{\varepsilon
}+a_{\varepsilon})\nonumber\\
& \, -\int_{u_{\varepsilon}\in(0,1)}g^{1}(u_{\varepsilon})\ln(1-u_{\varepsilon
}+a_{\varepsilon})\text{.}\nonumber
\end{align}
Since $a_{\varepsilon}\searrow0$ as $\varepsilon\rightarrow0$, we may assume -
without loss of generality - that $0<a_{\varepsilon}<1/2$ for $\varepsilon$
small enough. So $\ln(s+a_{\varepsilon})\leq0$ for $s\in(0,1/2)$ and
$\ln(1-s+a_{\varepsilon})\leq0$ for $s\in(1/2,1).$ Hence, thanks to
(\ref{ipo:g5}), we have $-g^{1}(0)\ln(s+a_{\varepsilon})\geq0$ for
$s\in(0,1/2).$ Furthermore, (\ref{ipo:g3bis}) yields $\left\vert g^{1}%
(s)\ln(s+a_{\varepsilon})\right\vert \leq C$ for $s\in(1/2,1)$. Finally,
thanks to (\ref{ipo:gLip}), we obtain%
\begin{align}
\int_{u_{\varepsilon}\in(0,1)}g^{1}(u_{\varepsilon})\ln(u_{\varepsilon
}+a_{\varepsilon})  \leq& \int_{u_{\varepsilon}\in(0,1/2)}g^{1}%
(u_{\varepsilon})\ln(u_{\varepsilon}+a_{\varepsilon})\label{stimaimp2}\\
&  +\int_{u_{\varepsilon}\in(1/2,1)}\left\vert g^{1}(u_{\varepsilon}%
)\ln(u_{\varepsilon}+a_{\varepsilon})\right\vert \nonumber\\
 \leq &\int_{u_{\varepsilon}\in(0,1/2)}\left(  g^{1}(u_{\varepsilon}%
)-g^{1}(0)\right)  \ln(u_{\varepsilon}+a_{\varepsilon})+C\nonumber\\
\leq &\,-\int_{u_{\varepsilon}\in(0,1/2)}Lu_{\varepsilon}\ln(u_{\varepsilon
}+a_{\varepsilon})+C\nonumber\\
\leq&\,-\int_{u_{\varepsilon}\in(0,1/2)}L\left(  u_{\varepsilon}%
+a_{\varepsilon}\right)  \ln(u_{\varepsilon}+a_{\varepsilon})+C\leq C\nonumber
\end{align}
where $L$ is the lipschitz constant for $g$. The proof of
\begin{equation}
-\int_{u_{\varepsilon}\in(0,1)}g^{1}(u_{\varepsilon})\ln(1-u_{\varepsilon
}+a_{\varepsilon})\leq C \label{stimaimp3}%
\end{equation}
is analogous. Integrating (\ref{ret}) in time, we obtain%
\begin{align}\nonumber
\left\Vert \left(  \mu_{\varepsilon}\right)  ^{1/2}\nabla v_{\varepsilon
}\right\Vert _{L^{2}(Q)}  &  \leq C\text{,}\\
\left\vert \int_{\Omega}f_{\varepsilon}(u_{\varepsilon})\right\vert  &  \leq
C\text{.} \label{FTR}%
\end{align}
Therefore (see (\ref{eq:CHrel}))%
\[
\left\Vert u_{\varepsilon}^{\prime}\right\Vert _{L^{2}(0,T,\left(
H^{1}(\Omega)\right)  ^{\ast})}\leq C\text{.}%
\]

Using compactness results as in Lemma \ref{probdisc} we obtain (for a not
relabeled subsequence) that there exists $
u\in H^{1}(0,T,H^{1}(\Omega),(H^{1}(\Omega))^{\ast})\cap L^{\infty}%
(0,T,L^{2}(\Omega))$
such that
\[
u_{\varepsilon}\rightharpoonup u~\text{\ weakly in }L^{2}(0,T,H^{1}%
(\Omega))\text{,}%
\]%
\[
u_{\varepsilon}\rightharpoonup u~\text{\ weakly* in }L^{\infty}(0,T,L^{2}%
(\Omega))\text{,}%
\]%
\begin{equation}
u_{\varepsilon}\rightarrow u~\text{\ strongly in }L^{2}(0,T,L^{2}%
(\Omega))\text{ and a.e.~in }Q\text{,} \label{q}%
\end{equation}%
\[
u_{\varepsilon}^{\prime}\rightharpoonup u^{\prime}~\text{\ weakly in }%
L^{2}(0,T,\left(  H^{1}(\Omega)\right)  ^{\ast})\text{,}%
\]%
\[
g^{1}(u_{\varepsilon})\rightarrow g^{1}(u)\text{ pointwise }a.e.\text{ in
}Q\text{.}%
\]
Furthermore, (\ref{K3}) yields%
\begin{equation}
w_{\varepsilon}\rightarrow w=K\ast(1-2u)\text{ in }L^{2}(0,T,H^{1}%
(\Omega))\text{.} \label{ps conv}%
\end{equation}
Thanks to (\ref{yui}) we get%
\begin{equation}
\mu_{\varepsilon}(u_{\varepsilon})\rightarrow\mu(u)\text{ }a.e.\text{ in
}Q\text{,} \label{pippo}%
\end{equation}
therefore
\[
\mu_{\varepsilon}(u_{\varepsilon})\nabla w_{\varepsilon}\rightarrow
\mu(u)\nabla w\text{ in }L^{2}(Q)\text{.}%
\]
Indeed
\begin{align*}
\left\Vert \mu_{\varepsilon}(u_{\varepsilon})\nabla w_{\varepsilon}%
-\mu(u)\nabla w\right\Vert _{L^{2}(\Omega)}  &  \leq\left\Vert \left(
\mu_{\varepsilon}(u_{\varepsilon})-\mu(u)\right)  \nabla w_{\varepsilon
}\right\Vert _{L^{2}(\Omega)}\\
+\left\Vert \mu(u)\left(  \nabla w_{\varepsilon}-\nabla w\right)  \right\Vert
_{L^{2}(\Omega)}  &  \leq\left\Vert \mu_{\varepsilon}(u_{\varepsilon}%
)-\mu(u)\right\Vert _{L^{2}(\Omega)}\left\Vert \nabla w_{\varepsilon
}\right\Vert _{L^{2}(\Omega)}\\
&  +\left\Vert \mu(u)\right\Vert _{L^{2}(\Omega)}\left\Vert \nabla
w_{\varepsilon}-\nabla w\right\Vert _{L^{2}(\Omega)}.
\end{align*}
Using (\ref{mu-eps-lim}), (\ref{ps conv}), (\ref{pippo}) and dominated
convergence Theorem we have
\[
\left\Vert \mu_{\varepsilon}(u_{\varepsilon})-\mu(u)\right\Vert _{L^{2}%
(\Omega)}\left\Vert \nabla w_{\varepsilon}\right\Vert _{L^{2}(\Omega
)}\rightarrow0\text{ and }
\left\Vert \mu(u)\right\Vert _{L^{2}(\Omega)}\left\Vert \nabla w_{\varepsilon
}-\nabla w\right\Vert _{L^{2}(\Omega)}\rightarrow0\hbox{ for }\varepsilon\rightarrow0.
\]
Now, we can pass to the limit $\varepsilon\rightarrow0$ in (\ref{eq:CHrel}%
)-(\ref{eq:icrel}) and obtain $u$ solution to (\ref{eq:CH})-(\ref{eq:ic}) with
$g^{1}$ instead of $g$.
In order to prove Theorem \ref{Teorema principale}, we are only left to show
that%
\[
0\leq u\leq1
\]
holds. From (\ref{def:mu-epsilon}) we have that $\mu_{\varepsilon
}(s)=\varepsilon$ for every $s<0$ and $s>1$. Hence, as consequence of
(\ref{we giovane}), $f_{\varepsilon}^{\prime\prime}(s)=\frac{1+2a_{\varepsilon
}}{\varepsilon}$ for every $s<0$ and $s>1.$ Therefore
\[
f_{\varepsilon}^{\prime}(s)\geq\frac{1+2a_{\varepsilon}}{\varepsilon
}(s-1)+f_{\varepsilon}^{\prime}(1)\geq\frac{1+2a_{\varepsilon}}{\varepsilon
}(s-1).
\]
Finally
\[
f_{\varepsilon}(s)\geq\frac{1+2a_{\varepsilon}}{2\varepsilon}(s-1)^{2}%
+f_{\varepsilon}(1)\text{.}%
\]
Likewise, we can prove
\[
f_{\varepsilon}(s)\geq\frac{1+2a_{\varepsilon}}{2\varepsilon}s^{2}%
+f_{\varepsilon}(0)\text{.}%
\]
So, thanks to (\ref{FTR}),
\begin{align*}
\int_{u_{\varepsilon}>1}(u_{\varepsilon}-1)^{2}  &  \leq\frac{2\mu
_{\varepsilon}(1)}{1+2a_{\varepsilon}}\int_{u_{\varepsilon}>1}\left\vert
f_{\varepsilon}(u_{\varepsilon})\right\vert -\frac{2\mu_{\varepsilon}%
(1)}{1+2a_{\varepsilon}}\int_{u_{\varepsilon}>1}\left\vert f_{\varepsilon
}(1)\right\vert \\
&  \leq\frac{\mu_{\varepsilon}(1)}{1+2a_{\varepsilon}}\left(  C-2\int
_{u_{\varepsilon}>1}\left\vert f_{\varepsilon}(1)\right\vert \right)  .
\end{align*}
Using (\ref{def:f}) and taking into account that $\frac{\mu_{\varepsilon}%
(1)}{1+2a_{\varepsilon}}=o(1)$, $f_{\varepsilon}(1)=o(1)$ for $\varepsilon
\rightarrow0$ we get
\[
\int_{u>1}\left(  u-1\right)  ^{2}=0\text{.}%
\]
Hence $u\leq1$ $a.e.$ in $Q$. The proof of $u\geq0$ $a.e.$ in $Q$ is analogous.

This yields $g^{1}(u)=g(u)$, so $u$ is solution to (\ref{eq:CH})-(\ref{eq:ic})
for every $g$ that satisfies (\ref{G2})-(\ref{ipo:g5}).

\section{Regularity\label{RRRR}}

Section \ref{RRRR} is devoted to the proofs of Theorem \ref{teorego}, Corollary
\ref{Cororego} and Theorem \ref{teo v reg}. Our proofs of Theorem \ref{teorego}
and Corollary \ref{Cororego} follows the guide-line of proof of Theorem 2.2 in
\cite{LP2}, where the same results are proved in the case $g=0$.

\bigskip

\subsection{Proof of Theorem \ref{teorego}}

The following Lemmas \ref{lemma1}-\ref{lemma4} are preliminary results needed
in order to prove Theorem \ref{teorego}.

\begin{lemma}
\label{lemma1}Let the assumption of Theorem \ref{teorego} be satisfied. Then
the solution $u$ to (\ref{eq:CH})-(\ref{eq:ic}) in the sense of Definition
\ref{def1} satisfies%
\begin{equation}
u_{t}\in L^{\infty}(0,T,\left(  H^{1}(\Omega)\right)  ^{\ast})\cap
L^{2}(0,T,L^{2}(\Omega))\text{.} \label{lem1}%
\end{equation}

\end{lemma}

\begin{proof}
First we observe that, thanks to (\ref{u0h2}), $u_{t}(0)\in\left(
H^{1}(\Omega)\right)  ^{\ast}$. From (\ref{ops}), we have 
$$\bar{u}_{t}%
=\int_{\Omega}g(x,t,u(x,t))\,dx\quad \hbox{for a.a.~}t\in\lbrack0,T].$$ 
Hence, as
consequence of (\ref{ipo:Gi3}), since $\Omega$ is bounded, we have%
\begin{equation}
\bar{u}_{t}\in L^{\infty}(0,T)\text{ and }\left\Vert \bar{u}_{t}\right\Vert
_{L^{\infty}(0,T)}\leq C\text{.} \label{medLim}%
\end{equation}
Denote $H_{0}^{-1}(\Omega)=\left(  H_{0}^{1}(\Omega)\right)  ^{\ast}$. In
order to show (\ref{lem1}) we only have to prove
\[
U_{t}=u_{t}-\bar{u}_{t}\in L^{\infty}(0,T,H_{0}^{-1}(\Omega))\cap
L^{2}(0,T,L^{2}(\Omega)).
\]
It is not hard to show that $U_{t}\in H_{0}^{-1}(\Omega)$ for $a.a.~t\in
\lbrack0,T]$. Let $\Delta_{N}:H_{0}^{1}(\Omega)\rightarrow H_{0}^{-1}(\Omega)$
be the realization of the Laplacian with the Neumann boundary conditions. Henceforth we will
proceed formally: the proof can be made exact by approximation of the
$t$-derivative by the corresponding quotient. Differentiating equation
(\ref{CHWEAK}) with respect to $t$ and taking the scalar product with
$\Delta_{N}^{-1}U_{t}$ we can prove the following%
\begin{align}
\frac{d}{dt}\left\Vert U_{t}\right\Vert _{H_{0}^{-1}(\Omega)}^{2}  &
=2(U_{tt},U_{t})_{H_{0}^{-1}(\Omega)}\label{tizio}\\
&  =2(\nabla\Delta_{N}^{-1}U_{tt},\nabla\Delta_{N}^{-1}U_{t})_{L^{2}(\Omega
)}=-2\left(  U_{tt},\Delta_{N}^{-1}U_{t}\right) \nonumber
\end{align}
and, using (\ref{eq:bc}) and (\ref{u0h3}),
\begin{align}
\left\langle \nabla(\mu\nabla v)_{t},\Delta_{N}^{-1}U_{t}\right\rangle  &
=-\left(  \left(  \mu\nabla w\right)  _{t},\nabla\Delta_{N}^{-1}U_{t}\right)
_{L^{2}(\Omega)}-\left(  \nabla u_{t},\nabla\Delta_{N}^{-1}U_{t}\right)
_{L^{2}(\Omega)}\label{cai}\\
&  =-\left(  \left(  \mu\nabla w\right)  _{t},\nabla\Delta_{N}^{-1}%
U_{t}\right)  _{L^{2}(\Omega)}-\left(  \nabla U_{t},\nabla\Delta_{N}^{-1}%
U_{t}\right)  _{L^{2}(\Omega)}\nonumber\\
&  =-\left(  \left(  \mu\nabla w\right)  _{t},\nabla\Delta_{N}^{-1}%
U_{t}\right)  _{L^{2}(\Omega)}+\left\Vert U_{t}\right\Vert _{L^{2}(\Omega
)}^{2}\text{.}\nonumber
\end{align}
Hence, adding together (\ref{tizio}) and (\ref{cai}), we obtain
\begin{align}
\frac{1}{2}\frac{d}{dt}\left\Vert U_{t}\right\Vert _{H_{0}^{-1}(\Omega)}%
^{2}+\left\Vert U_{t}\right\Vert _{L^{2}(\Omega)}^{2}  &  =-\left(
U_{tt},\Delta_{N}^{-1}U_{t}\right)  _{L^{2}(\Omega)}\label{ippo}\\
+\left(  \left(  \mu\nabla w\right)  _{t},\nabla\Delta_{N}^{-1}U_{t}\right)
_{L^{2}(\Omega)}  &  +\left\langle \nabla(\nabla u+\mu\nabla w)_{t},\Delta
_{N}^{-1}U_{t}\right\rangle \text{.}\nonumber
\end{align}
Starting from (\ref{eq:CH}) and differentiating with respect to $t$ we obtain
$U_{tt}=u_{tt}-\bar{u}_{tt}=u_{tt}-\int_{\Omega}\partial_{t}\left(  g\left(
u\right)  \right)  =\nabla u_{t}+\left(  \mu\nabla w\right)  _{t}+\partial
_{t}(g\left(  u\right)  )-\int_{\Omega}\partial_{t}\left(  g\left(  u\right)
\right)  $. So, thanks to (\ref{ippo}), we have
\begin{align}
\frac{1}{2}\frac{d}{dt}\left\Vert U_{t}\right\Vert _{H_{0}^{-1}(\Omega)}%
^{2}+\left\Vert U_{t}\right\Vert _{L^{2}(\Omega)}^{2}  &  =\left(  \left(
\mu\nabla w\right)  _{t},\nabla\Delta_{N}^{-1}U_{t}\right) \label{e}\\
&  -\left(  \partial_{t}\left(  g\left(  u\right)  \right)  ,\Delta_{N}%
^{-1}U_{t}\right)  +\left(  \int_{\Omega}\partial_{t}\left(  g\left(
u\right)  \right)  ,\Delta_{N}^{-1}U_{t}\right)  \text{.}\nonumber
\end{align}
Using (\ref{K3}), (\ref{def:mu}) and (\ref{xxx}) we estimate%
\begin{align*}
\left\Vert \mu_{t}\nabla w+\mu\nabla w_{t}\right\Vert _{L^{2}(\Omega)}  &
\leq\left\Vert u_{t}(1-2u)\nabla w\right\Vert _{L^{2}(\Omega)}+\left\Vert
\nabla(K\ast(1-2u))_{t}\right\Vert _{L^{2}(\Omega)}\\
&  \leq\left\Vert u_{t}\right\Vert _{L^{2}(\Omega)}\left\Vert \nabla
w\right\Vert _{L^{\infty}(\Omega)}+\left\Vert \nabla(K\ast u_{t})\right\Vert
_{L^{2}(\Omega)}\\
&  \leq C\left\Vert u_{t}\right\Vert _{L^{2}(\Omega)}\text{.}%
\end{align*}
Hence, using (\ref{medLim}), we get
\begin{align}
(\mu_{t}\nabla w+\mu\nabla w_{t},\nabla\Delta_{N}^{-1}U_{t})  &  \leq
C\left\Vert u_{t}\right\Vert _{L^{2}(\Omega)}\left\Vert \nabla\Delta_{N}%
^{-1}U_{t}\right\Vert _{L^{2}(\Omega)}\label{ee}\\
&  \leq C(1+\left\Vert U_{t}\right\Vert _{L^{2}(\Omega)})\left\Vert
U_{t}\right\Vert _{H_{0}^{-1}(\Omega)}\nonumber\\
&  \leq\frac{1}{2}\left\Vert U_{t}\right\Vert _{L^{2}(\Omega)}^{2}+C\left\Vert
U_{t}\right\Vert _{H_{0}^{-1}(\Omega)}^{2}+D\left\Vert U_{t}\right\Vert
_{H_{0}^{-1}(\Omega)}\text{.}\nonumber
\end{align}
Assumptions (\ref{ipo:gLip}) and (\ref{ipo:gReg3}) together with  (\ref{medLim}) yield
\begin{align}
\left\vert \left(  \partial_{t}g(u),\Delta_{N}^{-1}U_{t}\right)
_{L^{2}(\Omega)}\right\vert  &  =\left\vert \left(  g_{u}(u)u_{t},\Delta
_{N}^{-1}U_{t}\right)  _{L^{2}(\Omega)}\right. \label{uuuu}\\
&  \left.  +\left(  g_{t}(u),\Delta_{N}^{-1}U_{t}\right)  _{L^{2}(\Omega
)}\right\vert \nonumber\\
&  \leq L\left\Vert u_{t}\right\Vert _{L^{2}\left(  \Omega\right)  }\left\Vert
\Delta_{N}^{-1}U_{t}\right\Vert _{L^{2}(\Omega)}+C\left\Vert \Delta_{N}%
^{-1}U_{t}\right\Vert _{L^{2}(\Omega)}\nonumber\\
&  \leq\frac{1}{4}\left\Vert U_{t}\right\Vert _{L^{2}(\Omega)}^{2}+C\left\Vert
\Delta_{N}^{-1}U_{t}\right\Vert _{L^{2}(\Omega)}^{2}\nonumber\\
&  +C\left\Vert \Delta_{N}^{-1}U_{t}\right\Vert _{L^{2}(\Omega)}%
+C\text{.}\nonumber
\end{align}
Since $\Delta_{N}^{-1}U_{t}\in H_{0}^{1}(\Omega)$, thanks to Poincar\'{e}'s
inequality, we have $\left\Vert \Delta_{N}^{-1}U_{t}\right\Vert _{L^{2}%
(\Omega)}\leq C\left\Vert \nabla\Delta_{N}^{-1}U_{t}\right\Vert _{L^{2}%
(\Omega)}=C\left\Vert U_{t}\right\Vert _{H_{0}^{-1}(\Omega)}$. From
(\ref{uuuu}) it follows
\begin{align}
\left\vert \left(  \partial_{t}g(u),\Delta_{N}^{-1}U_{t}\right)
_{L^{2}(\Omega)}\right\vert  &  \leq\frac{1}{4}\left\Vert U_{t}\right\Vert
_{L^{2}(\Omega)}^{2}\label{eee}\\
&  +C\left\Vert U_{t}\right\Vert _{H_{0}^{-1}(\Omega)}+D\left\Vert
U_{t}\right\Vert _{H_{0}^{-1}(\Omega)}^{2}\text{. }\nonumber
\end{align}
Similarly we get %
\begin{align}
\left\vert \left(  \int_{\Omega}\partial_{t}g(u),\Delta_{N}^{-1}U_{t}\right)
_{L^{2}(\Omega)}\right\vert  &  \leq\frac{1}{8}\left\Vert U_{t}\right\Vert
_{L^{2}(\Omega)}^{2}\label{eeee}\\
&  +\left\vert \Omega\right\vert C\left\Vert U_{t}\right\Vert _{H_{0}%
^{-1}(\Omega)}+D\left\vert \Omega\right\vert ^{2}\left\Vert U_{t}\right\Vert
_{H_{0}^{-1}(\Omega)}^{2}.\nonumber
\end{align}
Finally, (\ref{e}), (\ref{ee}), (\ref{eee}) and (\ref{eeee}) yield
\begin{align*}
\frac{1}{2}\frac{d}{dt}\left\Vert U_{t}\right\Vert _{H_{0}^{-1}(\Omega)}%
^{2}+\frac{1}{8}\left\Vert U_{t}\right\Vert _{L^{2}(\Omega)}^{2}  &  \leq
C\left\Vert U_{t}\right\Vert _{H_{0}^{-1}(\Omega)}+C\left\Vert U_{t}%
\right\Vert _{H_{0}^{-1}(\Omega)}^{2}\\
&  \leq C+C\left\Vert U_{t}\right\Vert _{H_{0}^{-1}(\Omega)}^{2}\text{.}%
\end{align*}
Integrating in time and using Gronwall's Lemma \ we get $\left\Vert
U_{t}\right\Vert _{L^{\infty}(0,T,H_{0}^{-1}(\Omega))}+\left\Vert
U_{t}\right\Vert _{L^{2}(0,T,L^{2}(\Omega))}\leq C$ and so (recalling
(\ref{medLim})) that
\[
\left\Vert u_{t}\right\Vert _{L^{\infty}(0,T,\left(  H^{1}(\Omega)\right)
^{\ast})}+\left\Vert u_{t}\right\Vert _{L^{2}(0,T,L^{2}(\Omega))}\leq
C\text{.}%
\]
This concludes the proof of the lemma.
\end{proof}

\begin{lemma}
\label{lemma2}Let the assumptions of Theorem \ref{teorego} be satisfied. Then
the solution$~u$ to (\ref{eq:CH})-(\ref{eq:ic}) in the sense of Definition
\ref{def1} satisfies
\[
u_{t}\in L^{\infty}(0,T,L^{2}(\Omega))\cap L^{2}(0,T,H^{1}\left(
\Omega\right)  ).
\]

\end{lemma}

\begin{proof}
Thanks to (\ref{medLim}) and to the fact that $\nabla\bar{u}_{t}=0$,  we need only to prove
that $U_{t}\in L^{\infty}(0,T,L^{2}(\Omega))\cap L^{2}(0,T,H^{1}\left(
\Omega\right)  )$. We proceed as in Lemma \ref{lemma1}, but after
differentiating in time we multiply by $U_{t}$ (instead of $\Delta_{N}%
^{-1}U_{t}$). After integrating by parts with respect to $t$ we obtain%
\begin{align*}
\frac{1}{2}\frac{d}{dt}\left\Vert U_{t}\right\Vert _{L^{2}(\Omega)}^{2}  &
=-(\mu_{t}\nabla w+\mu\nabla w_{t}+\nabla U_{t},\nabla U_{t})_{L^{2}(\Omega
)}\\
&  \quad+\left(  \partial_{t}g(u),U_{t}\right)  _{L^{2}(\Omega)}-\left(
\int_{\Omega}\partial_{t}g(u),U_{t}\right)  _{L^{2}(\Omega)}\\
&  \leq C\left\Vert u_{t}\right\Vert _{L^{2}(\Omega)}\left\Vert \nabla
U_{t}\right\Vert _{L^{2}(\Omega)}-\left\Vert \nabla U_{t}\right\Vert
_{L^{2}(\Omega)}^{2}+D\left\Vert U_{t}\right\Vert _{L^{2}(\Omega)}\\
&  \leq C\left\Vert U_{t}\right\Vert _{L^{2}(\Omega)}+D\left\Vert
U_{t}\right\Vert _{L^{2}(\Omega)}^{2}-\frac{1}{2}\left\Vert \nabla
U_{t}\right\Vert _{L^{2}(\Omega)}^{2}\text{.}%
\end{align*}
Integrating with respect to $t$, we get%
\[
\left\Vert U_{t}(t)\right\Vert _{L^{2}(\Omega)}^{2}+\int_{0}^{t}\left\Vert
\nabla U_{t}\right\Vert _{L^{2}(\Omega)}^{2}\leq\left\Vert U_{t}(0)\right\Vert
_{L^{2}(\Omega)}^{2}+C\int_{0}^{t}\left\Vert U_{t}\right\Vert _{L^{2}(\Omega
)}^{2}\text{.}%
\]
We remark that $\left\Vert U_{t}(0)\right\Vert _{L^{2}(\Omega)}^{2}$ is bounded
(thanks to (\ref{u0h2})). This, coupled with Lemma \ref{lemma1} and Gronwall's
Lemma, yields%
\[
\left\Vert u_{t}\right\Vert _{L^{\infty}(0,T,L^{2}(\Omega))}+\left\Vert
u_{t}\right\Vert _{L^{2}(0,T,H^{1}(\Omega))}\leq C\text{.}%
\]
This concludes the proof of the lemma. 
\end{proof}

\begin{lemma}
\label{lemma3}Let the assumptions of Theorem \ref{teorego} be satisfied. Then
the solution $u$ to (\ref{eq:CH})-(\ref{eq:ic}) in the sense of Definition
\ref{def1} satisfies
\begin{equation}
\nabla u\in L^{\infty}(0,T,L^{2}(\Omega)). \label{lolo}%
\end{equation}

\end{lemma}

\begin{proof}
Thanks to (\ref{rego}) and Lemma \ref{lemma2} we have $\nabla u\in
H^{1}(0,T,L^{2}(\Omega))$ and hence (\ref{lolo}) follows.
\end{proof}

\begin{lemma}
\label{lemma4}Let the assumptions of Theorem \ref{teorego} be satisfied. Then
the solution$~u$ to (\ref{eq:CH})-(\ref{eq:ic}) in the sense of Definition
\ref{def1} satisfies $u\in L^{\infty}(0,T,H^{2}(\Omega))$.
\end{lemma}

\begin{proof}
We rewrite equation (\ref{CHWEAK}) in the form:%
\begin{align*}
\left\langle u_{t},\psi\right\rangle  &  =\left\langle \Delta u,\psi
\right\rangle +\left(  (1-2u)\nabla u\nabla w+\mu\Delta w+g(u),\psi\right) \\
\forall\psi &  \in H^{1}(\Omega)~\hbox{ and a.a.}~t\in(0,T).
\end{align*}
We remark that $u_{t}\in L^{\infty}(0,T,L^{2}(\Omega))$ thanks to Lemma
\ref{lemma2}; $(1-2u)\nabla u\nabla w$ $\in L^{\infty}(0,T,L^{2}(\Omega))$ as
a consequence of Lemma \ref{lemma3}; $\mu\Delta w\in L^{\infty}(0,T,L^{2}%
(\Omega))$ because of (\ref{K4}) and Lemma \ref{lemma3}. From (\ref{ipo:Gi3})
follows $g(u)\in L^{\infty}(0,T,L^{2}(\Omega))$. So%
\begin{equation}
\left\langle \Delta u,\psi\right\rangle =\left(  \xi,\psi\right)  ~\forall
\psi\in H^{1}(\Omega)~\text{\ for }a.a.~t\in(0,T) \label{tro}%
\end{equation}
where%
\begin{equation}
\xi=u_{t}+(1-2u)\nabla u\nabla w+\mu\Delta w+g(u)\in L^{\infty}(0,T,L^{2}%
(\Omega)). \label{cento39}%
\end{equation}
Thanks to (\ref{xxx}) and Lemma \ref{lemma3}, we have
\[
u\in L^{\infty}(0,T,H^{1}(\Omega))\cap L^{\infty}(Q)\text{.}%
\]
So, through (\ref{K3}) and (\ref{K4}) we get
\begin{align*}
&w  \in L^{\infty}(0,T,H^{2}(\Omega))\cap L^{\infty}\left(  0,T,W^{1,\infty
}(\Omega)\right)  \text{ and }
\nabla w    \in L^{\infty}(0,T,H^{1}(\Omega))\cap L^{\infty}\left(  Q\right)
\text{.}%
\end{align*}
Furthermore, since $\partial\Omega\in Lip$, then $n\in L^{\infty}%
(\partial\Omega)$, where $n$ denotes the outer unit normal on $\partial\Omega
$. Hence (see \cite{BG}, Theorem 2.7.4), we have
\begin{equation}
\frac{\partial w}{\partial n}\in L^{\infty}\left(  0,T,H^{1/2}(\partial
\Omega)\right)  \cap L^{\infty}\left(  0,T,L^{\infty}\left(  \partial
\Omega\right)  \right)  \text{.} \label{tu}%
\end{equation}
Thanks to (\ref{xxx}) and Lemma \ref{lemma3}, $\nabla\mu(u)=(1-2u)\nabla u\in
L^{\infty}(0,T,L^{2}\left(  \Omega\right)  )$. Thus
\begin{equation}
\mu\left(  u\right)  \in L^{\infty}(0,T,H^{1/2}\left(  \Omega\right)
)\text{.} \label{tutu}%
\end{equation}
Recalling that $0\leq u\leq1$ and $0\leq\mu\left(  s\right)  \leq1$ for every
$s\in\lbrack0,1]$ we can extend $\mu$ so that $0\leq\mu\left(  s\right)
\leq1$ for every $s\in%
\mathbb{R}
$. Hence,%
\begin{equation}
\mu\left(  u\right)  \in L^{\infty}\left(  0,T,L^{\infty}\left(
\partial\Omega\right)  \right)  \text{.} \label{tututu}%
\end{equation}
Combining (\ref{tu}), (\ref{tutu}) and (\ref{tututu}) we obtain%
\begin{equation}
\mu(u)\frac{\partial w}{\partial n}\in L^{\infty}\left(  0,T,H^{1/2}%
(\partial\Omega)\right)  \text{.} \label{tr}%
\end{equation}
From (\ref{eq:bc}) follows%
\[
\frac{\partial u}{\partial n}=n\cdot\mu\nabla w=\mu(u)\frac{\partial
w}{\partial n}~a.e.~\text{on }\partial\Omega,
\]
and so, thanks to (\ref{tr}),%
\begin{equation}
\frac{\partial u}{\partial n}\in L^{\infty}\left(  0,T,H^{1/2}(\partial
\Omega)\right)  \text{.} \label{lop}%
\end{equation}
Finally, using an elliptic regularity theorem (Theorem \ref{elliptic rego} in
the Appendix), we get
\[
u\in H^{2}(\Omega)\text{ for a.a.}~t\in (0,T)
\]
and
\[
\left\Vert u\right\Vert _{H^{2}(\Omega)}\leq C\left(  \left\Vert u\right\Vert
_{L^{2}(\Omega)}+\left\Vert \xi\right\Vert _{L^{2}(\Omega)}+\left\Vert
\frac{\partial u}{\partial n}\right\Vert _{H^{1/2}(\partial\Omega)}\right)
\text{ for }a.a.~t\in\lbrack0,T]\text{.}%
\]
Combining $u\in L^{\infty}(0,T,L^{2}(\Omega))$, (\ref{cento39}) and
(\ref{lop}) we obtain
\[
u\in L^{\infty}(0,T,H^{2}\left(  \Omega\right)  )\text{.}%
\] 
This concludes the proof of the lemma. 
\end{proof}

Theorem \ref{teorego} follows directly from Lemma \ref{lemma4}.

In order to prove Corollary \ref{Cororego} we proceed as follows. Since
\[
u\in L^{2}(0,T,H^{1}(\Omega)),
\]
we have that for $a.a.~T_{0}\in(0,T),$ $u(T_{0})\in H^{1}(\Omega)$. Hence, we
can prove Lemma \ref{lemma1} for the solution to (\ref{eq:CH})-(\ref{eq:bc})
on $[T_{0}-\varepsilon,T]$ where $0<\varepsilon<T_{0}/2$. Therefore, there
exists $s\in\lbrack T_{0}-\varepsilon,T_{0}]$ such that $\left\Vert
u_{t}(s)\right\Vert _{L^{2}(\Omega)}$ is finite. We now proceed as in Lemma
\ref{lemma2}, \ref{lemma3} and \ref{lemma4} working on the set $[s,T]$ and
choosing $u(s)$ as initial data and we get $u\in L^{\infty}(s,T,H^{2}%
(\Omega))$ and so $u\in L^{\infty}(T_{0},T,H^{2}(\Omega))$.

\bigskip

\subsection{Proof of Theorem \ref{teo v reg}}

We now prove Theorem \ref{teo v reg}. Let the assumption of Theorem
\ref{teo v reg} be satisfied. Let $u_{\varepsilon}$ be the solution of the
approximate problem $P_{\varepsilon}$.

We first prove, by applying an Alikakos' iteration argument as in \cite[Theorem 2.1]{BH2},
that the family of approximate solutions $u_{\varepsilon}$ is
uniformly bounded in $L^{\infty}\left(  \Omega\right)  $. To see this, let us
take $\psi=|u_{\varepsilon}|^{p-1}u_{\varepsilon}$ as test function in
(\ref{eq:CHrel}), where $p>1$. Then we get the following differential
identity:
\begin{align}
&  \frac{1}{p+1}\frac{d}{dt}\left\Vert u_{\varepsilon}\right\Vert
_{L^{p+1}(\Omega)}^{p+1}+p\int_{\Omega}\mu_{\varepsilon}(u_{\varepsilon
})f_{\varepsilon}^{\prime\prime}(u_{\varepsilon})|\nabla u_{\varepsilon}%
|^{2}|u_{\varepsilon}|^{p-1}\label{orco}\\
\nonumber
&+p\int_{\Omega}\mu_{\varepsilon}(u_{\varepsilon})\nabla w_{\varepsilon}\nabla
u_{\varepsilon}|u_{\varepsilon}|^{p-1}    =\int_{\Omega}g(u_{\varepsilon
})u_{\varepsilon}|u_{\varepsilon}|^{p-1}.
\end{align}
Actually, the above choice of test function would not be generally admissible.
Nevertheless, the argument can be made rigorous by means of a density
procedure, e.g., by first truncating the test function $|u_{\varepsilon
}|^{p-1}u_{\varepsilon}$ and then passing to the limit with respect to the
truncation parameter. By using (\ref{def:mu-epsilon}), (\ref{def:f-epsiolon})
we obtain
\begin{equation}
p\int_{\Omega}\mu_{\varepsilon}(u_{\varepsilon})f_{\varepsilon}^{\prime\prime
}(u_{\varepsilon})|\nabla u_{\varepsilon}|^{2}|u_{\varepsilon}|^{p-1}\geq
\frac{4p}{\left(  p+1\right)  ^{2}}c\int_{\Omega}\left\vert \nabla
|u_{\varepsilon}|^{\frac{p+1}{2}}\right\vert ^{2} \label{gigi}%
\end{equation}
where $c>0$ not depending on $\varepsilon$. Therefore, by combining
(\ref{orco}) with (\ref{gigi}) we deduce%
\begin{align}
&  \frac{1}{p+1}\frac{d}{dt}\left\Vert u_{\varepsilon}\right\Vert
_{L^{p+1}(\Omega)}^{p+1}+\frac{4p}{\left(  p+1\right)  ^{2}}c\int_{\Omega
}\left\vert \nabla|u_{\varepsilon}|^{\frac{p+1}{2}}\right\vert ^{2}%
+p\int_{\Omega}\mu_{\varepsilon}(u_{\varepsilon})\nabla w_{\varepsilon}\nabla
u_{\varepsilon}|u_{\varepsilon}|^{p-1}\label{after}\\
&  \leq\int_{\Omega}g(u_{\varepsilon})u_{\varepsilon}|u_{\varepsilon}%
|^{p-1}\leq C\int_{\Omega}|u_{\varepsilon}|^{p}\text{. }\nonumber
\end{align}
Starting from (\ref{after}) and using the fact that $|\mu_{\varepsilon
}(u_{\varepsilon})|\leq C$, we can argue exactly as in \cite[Proof of
Theorem 2.1]{BH2} in order to conclude that
\begin{equation}
\left\Vert u_{\varepsilon}\right\Vert _{L^{\infty}(\Omega)}\leq C\text{. }
\label{linftybound}%
\end{equation}
Hence we can choose $\psi=f_{\varepsilon}^{\prime\prime}(u_{\varepsilon
})f_{\varepsilon}^{\prime}(u_{\varepsilon})$ as test function in
(\ref{eq:CHrel}) and get%
\begin{align}
&  \frac{1}{2}\frac{d}{dt}\left\Vert f_{\varepsilon}^{\prime}(u_{\varepsilon
})\right\Vert _{L^{2}(\Omega)}^{2}+\int_{\Omega}\mu_{\varepsilon}\nabla
v_{\varepsilon}f_{\varepsilon}^{\prime\prime}(u_{\varepsilon})\nabla
(f_{\varepsilon}^{\prime}(u_{\varepsilon}))\label{step1}
+\int_{\Omega}\mu_{\varepsilon}\nabla v_{\varepsilon}f_{\varepsilon}%
^{\prime\prime\prime}(u_{\varepsilon})f_{\varepsilon}^{\prime}(u_{\varepsilon
})\nabla u_{\varepsilon}    =\int_{\Omega}g^{1}(u_{\varepsilon}%
)f_{\varepsilon}^{\prime\prime}(u_{\varepsilon})f_{\varepsilon}^{\prime
}(u_{\varepsilon}).
\end{align}
We now observe that, thanks to (\ref{we giovane}) and (\ref{propf}),
$f_{\varepsilon}^{\prime\prime}(s)f_{\varepsilon}^{\prime}(s)\leq0$ if $s<1/2$
and $f_{\varepsilon}^{\prime\prime}(s)f_{\varepsilon}^{\prime}(s)\geq0$ if
$s>1/2$. Thus, recalling (\ref{G5bis}), we have%
\begin{align*}
&\int_{u_{\varepsilon}<0}g^{1}(u_{\varepsilon})f_{\varepsilon}^{\prime\prime
}(u_{\varepsilon})f_{\varepsilon}^{\prime}(u_{\varepsilon})   \leq0\text{
and }
\int_{u_{\varepsilon}>1}g^{1}(u_{\varepsilon})f_{\varepsilon}^{\prime\prime
}(u_{\varepsilon})f_{\varepsilon}^{\prime}(u_{\varepsilon})   \leq0\text{.}%
\end{align*}
Furthermore, as a consequence of assumptions (\ref{ipo:gLip}) and
(\ref{ipo:g5}) and of (\ref{we giovane}), we get
\begin{align*}
\int_{0\leq u_{\varepsilon}\leq1/2}g^{1}(u_{\varepsilon})f_{\varepsilon
}^{\prime\prime}(u_{\varepsilon})f_{\varepsilon}^{\prime}(u_{\varepsilon})  &
\leq\int_{0\leq u_{\varepsilon}\leq1/2}\left(  g^{1}(u_{\varepsilon}%
)-g^{1}(0)\right)  f_{\varepsilon}^{\prime\prime}(u_{\varepsilon
})f_{\varepsilon}^{\prime}(u_{\varepsilon})\\
&  \leq\int_{0\leq u_{\varepsilon}\leq1/2}Lu_{\varepsilon}f_{\varepsilon
}^{\prime\prime}(u_{\varepsilon})f_{\varepsilon}^{\prime}(u_{\varepsilon})\\
&  \leq C\int_{0\leq u_{\varepsilon}\leq1/2}f_{\varepsilon}^{\prime
}(u_{\varepsilon})\leq C\left\Vert f_{\varepsilon}^{\prime}(u_{\varepsilon
})\right\Vert _{L^{2}(\Omega)}%
\end{align*}
and
\[
\int_{1/2\leq u_{\varepsilon}\leq1}g^{1}(u_{\varepsilon})f_{\varepsilon
}^{\prime\prime}(u_{\varepsilon})f_{\varepsilon}^{\prime}(u_{\varepsilon})\leq
C\left\Vert f_{\varepsilon}^{\prime}(u_{\varepsilon})\right\Vert
_{L^{2}(\Omega)}\text{.}%
\]
These inequalities yield%
\begin{equation}
\int_{\Omega}g^{1}(u_{\varepsilon})f_{\varepsilon}^{\prime\prime
}(u_{\varepsilon})f_{\varepsilon}^{\prime}(u_{\varepsilon})\leq C\left\Vert
f_{\varepsilon}^{\prime}(u_{\varepsilon})\right\Vert _{L^{2}(\Omega)}\text{.}
\label{step1.2}%
\end{equation}
Recalling (\ref{K3}) and (\ref{iop}) we obtain
\begin{align}
\int_{\Omega}\mu_{\varepsilon}\nabla v_{\varepsilon}f_{\varepsilon}%
^{\prime\prime}(u_{\varepsilon})\nabla(f_{\varepsilon}^{\prime}(u_{\varepsilon
}))  &  =\int_{\Omega}\mu_{\varepsilon}f_{\varepsilon}^{\prime\prime
}(u_{\varepsilon})\left\vert \nabla(f_{\varepsilon}^{\prime}(u_{\varepsilon
}))\right\vert ^{2}\label{step2}\\
&  \quad +\int_{\Omega}\mu_{\varepsilon}\nabla w_{\varepsilon}f_{\varepsilon
}^{\prime\prime}(u_{\varepsilon})\nabla(f_{\varepsilon}^{\prime}%
(u_{\varepsilon}))\nonumber\\
&  \geq1/2\int_{\Omega}\left\vert \nabla(f_{\varepsilon}^{\prime
}(u_{\varepsilon}))\right\vert ^{2}-C\text{.}\nonumber
\end{align}
We now observe that
\[
\mu_{\varepsilon}f_{\varepsilon}^{\prime\prime\prime}(u_{\varepsilon})\nabla
u_{\varepsilon}=\gamma_{\varepsilon}\nabla(f_{\varepsilon}^{\prime
}(u_{\varepsilon}))
\]
where
\[
\gamma_{\varepsilon}=f_{\varepsilon}^{\prime\prime\prime}(u_{\varepsilon}%
)\mu_{\varepsilon}^{2}\frac{1}{\mu_{\varepsilon}f_{\varepsilon}^{\prime\prime
}(u_{\varepsilon})}=f_{\varepsilon}^{\prime\prime\prime}(u_{\varepsilon}%
)\mu_{\varepsilon}^{2}\left(  1+o(\varepsilon)\right)  .
\]
It is not hard to show that $|\gamma_{\varepsilon}|\leq C$. Hence, by using
(\ref{linftybound}), (\ref{K3}) and the fact that $\gamma_{\varepsilon
}(s)f_{\varepsilon}^{\prime}(s)\geq0$ for every $s\in%
\mathbb{R}
$, we obtain%
\begin{align}
\int_{\Omega}\mu_{\varepsilon}\nabla v_{\varepsilon}f_{\varepsilon}%
^{\prime\prime\prime}(u_{\varepsilon})f_{\varepsilon}^{\prime}(u_{\varepsilon
})\nabla u_{\varepsilon}  &  =\int_{\Omega}\gamma_{\varepsilon}\left\vert
\nabla f_{\varepsilon}^{\prime}(u_{\varepsilon})\right\vert ^{2}%
f_{\varepsilon}^{\prime}(u_{\varepsilon})\label{step4}\\
& \quad +\int_{\Omega}\gamma_{\varepsilon}\nabla f_{\varepsilon}^{\prime
}(u_{\varepsilon})f_{\varepsilon}^{\prime}(u_{\varepsilon})\nabla
w_{\varepsilon}\nonumber\\
&  \geq1/4\int_{\Omega}\left\vert \nabla f_{\varepsilon}^{\prime
}(u_{\varepsilon})\right\vert ^{2}-C\int_{\Omega}\left\vert f_{\varepsilon
}^{\prime}(u_{\varepsilon})\right\vert ^{2}\left\vert \nabla w_{\varepsilon
}\right\vert ^{2}\nonumber\\
&  \geq1/4\int_{\Omega}\left\vert \nabla f_{\varepsilon}^{\prime
}(u_{\varepsilon})\right\vert ^{2}-C\int_{\Omega}\left\vert f_{\varepsilon
}^{\prime}(u_{\varepsilon})\right\vert ^{2}\text{.}\nonumber
\end{align}
Hence, combining (\ref{step1.2})-(\ref{step4}) together with (\ref{step1}), we
obtain
\[
\frac{1}{2}\frac{d}{dt}\left\Vert f_{\varepsilon}^{\prime}(u_{\varepsilon
})\right\Vert _{L^{2}(\Omega)}^{2}+\frac{1}{8}\int_{\Omega}\left\vert \nabla
f_{\varepsilon}^{\prime}(u_{\varepsilon})\right\vert ^{2}\leq C+C\left\Vert
f_{\varepsilon}^{\prime}\right\Vert _{L^{2}(\Omega)}^{2}\text{.}%
\]
Using Gronwall's Lemma, (\ref{u0h4}) and the fact that
\[
|f_{\varepsilon}^{\prime}(s)|\leq|f^{\prime}(s)|\text{, }\forall s\in
\lbrack0,1],\varepsilon>0\text{,}%
\]
we finally get%
\[
\left\Vert f_{\varepsilon}^{\prime}(u_{\varepsilon})\right\Vert _{L^{2}%
(0,T,H^{1}(\Omega))}\leq C\ \text{and }\left\Vert f_{\varepsilon}^{\prime
}(u_{\varepsilon})\right\Vert _{L^{\infty}(0,T,L^{2}(\Omega))}\leq C\text{.}%
\]
Thus, recalling $w_{\varepsilon}$ is bounded in $L^{\infty}(0,T,L^{2}%
(\Omega))\cap L^{2}(0,T,H^{1}(\Omega))$ independently of $\varepsilon$, we
obtain (\ref{regv}) and complete the proof of Theorem~\ref{teo v reg}.

\begin{remark}
Since $u_{\varepsilon}\rightarrow u~a.e.$ in $\Omega\times\lbrack0,T]$, thanks
to \cite[Theorem 8.3]{Ro},  $v_{\varepsilon}\rightarrow v=f^{\prime}(u)+w$
weakly in $L^{2}(0,T,H^{1}(\Omega))$. Hence $f^{\prime}(u)\in L^{2}(Q)$ and,
thus, $u\in(0,1)$ $a.e.$ in $\Omega\times\lbrack0,T]$. Furthermore, $u$ also
satisfies the weak formulation given by Definition \ref{def1} with
\[
\left\langle u_{t},\psi\right\rangle +(\mu(u)\nabla v,\nabla\psi
)=(g(u),\psi)\text{, \ \ }v=f^{\prime}(u)+w\text{,}%
\]
instead of (\ref{CHWEAK}).
\end{remark}

\section{Separation properties}
\label{Separazione}

This section is devoted to the study of separation from singularities of the
solution $u$ to (\ref{eq:CH})-(\ref{eq:ic}): we show that the solution of our
problem separates from the pure phases $0$ and $1$ after an arbitrary short
time $T_{0}$; more precisely we prove that for every $T_{0}\in(0,T)~$\ there
exists $k>0$ such that $k\leq u(x,t)\leq1-k$ for $a.a.~x\in\Omega,t\in\lbrack
T_{0},T]$. Moreover, if $u_{0}$ separates from $0$ and $1$ then $T_{0}=0$.

In \cite{LP2} Londen and Petzeltov\`{a} proved these results in case $g=0$.
Our proof follows the guide line of \cite[Theorem 2.1]{LP2}. The main
difference is due to the non-conservation of the quantity $\bar{u}(t)=\frac
{1}{\left\vert \Omega\right\vert }\int_{\Omega}u(x,t)$. We focus only on the
parts of the proof which differ from \cite[Theorem 2.1]{LP2}.

\begin{remark}
\label{ost}Since $0\leq u\leq1$, a necessary condition to $u$
being separated from $0$ and $1$ is 
\begin{equation}
0<\bar{u}(t)=\frac{1}{\left\vert \Omega\right\vert }\int_{\Omega
}u(x,t)dx<1~\forall t\in\lbrack0,T]\text{.} \label{urca}%
\end{equation}

\end{remark}

The following Lemmas \ref{lemma media1} and \ref{lemma media2} show that
$0<\bar{u}(t)<1~\forall t\in\lbrack0,T]$. Moreover they estimate the measure
of a level set of $u$ uniformly in $t$. These estimates will be used in
proving Theorem \ref{teo sep}.

\begin{lemma}
\label{lemma media1}Let $u$ be the weak solution to (\ref{eq:CH}%
)-(\ref{eq:ic}) in the sense of Definition \ref{def1}, let (\ref{u02}) and
(\ref{ipo:gLip}) be satisfied. Then, there exist $b_{0}>0$ and $c_{0}>0$ not
depending on $t$ such that,%
\begin{equation}
\left\vert \Omega_{1}^{t}\right\vert \geq c_{0}>0 \label{sep3}%
\end{equation}
where $\Omega_{1}^{t}=\{x\in\Omega:u(x,t)\geq b_{0}\}.$
\end{lemma}

\begin{proof}
Let us assume $\left\vert \Omega\right\vert =1$ for simplicity. We observe
$\bar{u}(t)-\bar{u}_{0}=\int_{0}^{t}\frac{d}{ds}\bar{u}(s)ds=\int_{0}%
^{t}\left\langle u^{\prime},1\right\rangle =\int_{0}^{t}(g,1)=\int_{0}^{t}%
\int_{\Omega}g$. Therefore
\begin{equation}
\bar{u}(t)=\bar{u}_{0}+\int_{0}^{t}\int_{\Omega}g(u)\text{.} \label{ops}%
\end{equation}
Thanks to (\ref{rego}) we have $u\in C(\left[  0,T\right]  ,L^{2}(\Omega))$.
Hence, the function $\bar{u}:t\in\lbrack0,T]\mapsto\bar{u}(t)\in\lbrack0,1]$ is
continuous. We first prove that there exists $c>0$ not depending on $t$ such
that%
\begin{equation}
\bar{u}(t)\geq c~\forall t\in\lbrack0,T]. \label{cc}%
\end{equation}
Suppose, by contradiction, that there exists $t^{\prime}\in(0,T]$ such that
$t^{\prime}=\min\{t\in\lbrack0,T]:\bar{u}(t)=0\}$. So $u(t^{\prime
})=0~a.a.~x\in\Omega$. Due to (\ref{ipo:g5}), $g(u(t^{\prime
}))\geq0$. As a consequence of (\ref{ipo:gLip}) \ there exists $L>0$ such that
$\left\vert g(x,t,s_{1})-g(x,t,s_{2})\right\vert \leq L\left\vert s_{1}%
-s_{2}\right\vert ~\forall(x,t,s_{i})\in\Omega\times\lbrack0,T]\times
\lbrack0,1], \, i\in\{1,2\}$. Hence, from (\ref{ops}), it follows
\begin{align*}
\bar{u}(t)  &  =\bar{u}_{0}+\int_{0}^{t}\int_{\Omega}g(u)\geq\bar{u}_{0}%
+\int_{0}^{t}\int_{\Omega}\left(  g(u(s))-g(u(t^{\prime}))\right)  ds\\
&  \geq\bar{u}_{0}-L\int_{0}^{t}\int_{\Omega}u(s)ds=\bar{u}_{0}-L\int_{0}%
^{t}\bar{u}(s)ds\quad\forall t\in [0,T]\text{.}%
\end{align*}
Then $\bar{u}(t)$ is bounded below by $\bar{u}(t)\geq\bar{u}_{0}\exp(-Lt)>0$
$\forall t\in\lbrack0,T]$. This contradicts $\bar{u}(t^{\prime})=0$. Hence,
(\ref{cc}) holds.

Set $b_{0}=\frac{1}{2}c$. Then%
\[
\left\vert \Omega_{1}^{t}\right\vert \geq\frac{1}{2}c=c_{0}>0\text{.}%
\]
Indeed, suppose, by contradiction, $\left\vert \Omega_{1}^{t}\right\vert
<\frac{1}{2}c$, then $c\leq\bar{u}(t)=\int_{\Omega_{1}^{t}}u+\int
_{\Omega\smallsetminus\Omega_{1}^{t}}u\leq\int_{\Omega_{1}^{t}}1+\int
_{\Omega\smallsetminus\Omega_{1}^{t}}\frac{c}{2}<\frac{c}{2}+\frac{c}%
{2}\left\vert \Omega\smallsetminus\Omega_{1}^{t}\right\vert $. Hence
$\left\vert \Omega\smallsetminus\Omega_{1}^{t}\right\vert >1$ which is a contradiction.
\end{proof}

\begin{lemma}
\label{lemma media2}Let $u$ be a weak solution to (\ref{eq:CH})-(\ref{eq:ic})
in the sense of Definition \ref{def1}, let (\ref{u02}) and (\ref{ipo:gLip}) be
satisfied. Then there exist $b_{0}>0$ and $c_{0}>0$ not depending on $t$ such
that%
\[
\left\vert \Omega_{2}^{t}\right\vert \geq c_{0}>0
\]
where $\Omega_{2}^{t}=\{x\in\Omega:u(x,t)\leq1-b_{0}\}$.
\end{lemma}

\begin{proof}
The proof is analogous to the proof of Lemma \ref{lemma media1}.
\end{proof}

The main result of this section is the following propositions.

\begin{proposition}
\label{separazione}Let assumption of Theorem \ref{teo sep} be satisfied. Then,
for every $T_{0}\in(0,T)$, there exists $k>0$ depending on $T_{0}$ and
$\bar{u}_{0}$ such that
\begin{equation}
k\leq u(x,t)\text{ for }a.a.\text{ }x\in\Omega\text{, }t\in\lbrack
T_{0},T]\,\text{.} \label{A}%
\end{equation}
Furthermore, if there exists $\tilde{k}>0$ such that
\begin{equation}
\tilde{k}\leq u_{0}~a.e.\text{ in }\Omega\text{,} \label{hp:u0bd}%
\end{equation}
then $T_{0}=0$.
\end{proposition}

\begin{proof}
In order to prove Proposition \ref{separazione} we follow the guide line of
\cite[Theorem 2.1]{LP2} . We show only the parts of the proof which differ from
\cite[Theorem 2.1]{LP2}. It is enough to show that $\ln(u(\cdot,t))$ is
bounded in $L^{\infty}(\Omega)$ by a constant depending on $T_{0}$ and
$\bar{u}_{0}$ for every $t\in\lbrack T_{0},T]$.

We prove first Proposition~\ref{separazione} assuming (\ref{hp:u0bd}). Without
loss of generality, thanks to Remark~\ref{superRemark}, we may assume that
$0<u(t)$ $a.e.$ in $\Omega$ for every $t\in\lbrack0,T].$

Denote%
\[
M_{r}(t)=\left\Vert \ln(u(\cdot,t))\right\Vert _{L^{r}(\Omega)}\text{ for
}r\in%
\mathbb{N}
\text{.}%
\]
We first derive a differential inequality for $M_{r}(t)$. Setting $r=1$ and
using (\ref{xxx}) and (\ref{CHWEAK}) we get
\begin{align*}
\frac{d}{dt}M_{1}(t)  &  =\frac{d}{dt}\int_{\Omega}\left(  -\ln(u)\right)
=-\left\langle u^{\prime},\frac{1}{u}\right\rangle \\
&  =\int_{\Omega}\nabla\left(  \frac{1}{u}\right)  \left(  \nabla u+\mu\nabla
w\right)  -\int_{\Omega}\frac{1}{u}g(u)\text{.}%
\end{align*}
From $g(0)\geq0$ (see (\ref{ipo:g5})) and $g(x,t,s)$ Lipschitz continuous in $s$ (see
(\ref{ipo:gLip})) follows the estimate
\begin{equation}
-\frac{g(u)}{u}\leq-\frac{g(u)}{u}+\frac{g(0)}{u}=-\frac{g(u)-g(0)}{u}%
\leq\frac{Lu}{u}=L\text{,} \label{impo}%
\end{equation}
where $L$ denotes the Lipschitz constant of $g$. So%
\begin{equation}
-\int_{\Omega}\frac{1}{u}g(u)\leq C\text{.} \label{stimett}%
\end{equation}
Using%
\begin{equation}
\left\vert \nabla\ln(u)\right\vert ^{2}=-\nabla u\nabla\left(  \frac{1}%
{u}\right)  \text{,} \label{id:h1}%
\end{equation}%
\begin{equation}
\frac{\nabla\ln(u)}{u}=-\nabla\left(  \frac{1}{u}\right)  \label{id:h2}%
\end{equation}
and (\ref{def:mu}), (\ref{rego}), (\ref{xxx}), (\ref{K3}) we prove the
following estimate
\begin{align}
\frac{d}{dt}M_{1}(t)  &  \leq-\int_{\Omega}\left\vert \nabla\ln(u)\right\vert
^{2}-\int_{\Omega}\frac{\nabla\ln(u)}{u}\mu\nabla w+C\label{sepp2}\\
&  =-\int_{\Omega}\left\vert \nabla\ln(u)\right\vert ^{2}-\int_{\Omega
}(1-u)\nabla\ln(u)\nabla w+C\nonumber\\
&  \leq-\frac{1}{2}\int_{\Omega}\left\vert \nabla\ln(u)\right\vert ^{2}%
+C\int_{\Omega}\left\vert \nabla w\right\vert ^{2}+C\nonumber\\
&  \leq-\frac{1}{2}\int_{\Omega}\left\vert \nabla\ln(u)\right\vert ^{2}%
+C\int_{\Omega}\left\vert u\right\vert ^{2}+C\nonumber\\
&  \leq-\frac{1}{2}\int_{\Omega}\left\vert \nabla\ln(u)\right\vert
^{2}+C\text{.}\nonumber
\end{align}
Let $b_{0}$, $\Omega_{1}^{t}$ and $c_{0}$ be as in Lemma \ref{lemma media1}.
Using Poincar\`{e}'s inequality (\ref{dis:Poincare}) (cf.~Theorem~\ref{teo poinc}
in the Appendix), (\ref{xxx}) and (\ref{sep3}), we obtain
\begin{align}
\int_{\Omega}\left\vert \nabla\ln(u)\right\vert ^{2}  &  \geq C\left\vert
\Omega_{1}^{t}\right\vert ^{2}\int_{\Omega}\left\vert \ln(u)-\frac
{1}{\left\vert \Omega_{1}^{t}\right\vert }\int_{\Omega_{1}^{t}}\ln
(u)\right\vert ^{2}\label{sepp4}\\
&  \geq C\left\vert \Omega_{1}^{t}\right\vert ^{2}\left[  \int_{\Omega
}\left\vert \ln(u)\right\vert ^{2}+\int_{\Omega}\frac{1}{\left\vert \Omega
_{1}^{t}\right\vert ^{2}}\left(  \int_{\Omega_{1}^{t}}\ln(u)\right)
^{2}\right. \nonumber\\
&  \left.  -\frac{2}{\left\vert \Omega_{1}^{t}\right\vert }\int_{\Omega}%
\ln(u)\int_{\Omega_{1}^{t}}\ln(u)\right] \nonumber\\
&  \geq C\left\vert \Omega_{1}^{t}\right\vert ^{2}\left(  \int_{\Omega
}\left\vert \ln(u)\right\vert ^{2}-\frac{2}{\left\vert \Omega_{1}%
^{t}\right\vert }\int_{\Omega}\ln(u)\int_{\Omega_{1}^{t}}\ln(u)\right)
\nonumber\\
&  \geq Cc_{0}^{2}\left(  \left(  \int_{\Omega}\left\vert \ln(u)\right\vert
\right)  ^{2}+\frac{2}{c_{0}}\int_{\Omega}\left\vert \ln(u)\right\vert
\int_{\Omega_{1}^{t}}\ln(b_{0})\right) \nonumber\\
&  \geq C\left(  M_{1}(t)\right)  ^{2}-CM_{1}(t)\nonumber\\
&  \geq C\left(  M_{1}(t)\right)  ^{2}-C\text{.}\nonumber
\end{align}
Combining together (\ref{sepp2}) and (\ref{sepp4}), we get%
\[
\frac{d}{dt}M_{1}(t)\leq-C_{1}\left(  M_{1}(t)\right)  ^{2}+C_{2}\text{.}%
\]
Proceeding as in \cite[Lemma 3.1]{LP2}, it is possible to prove that for every
$T_{0}\in\lbrack0,T]$ there exists $m_{1}$ depending on $\bar{u}_{0}$ and
$T_{0}$ such that
\begin{equation}
M_{1}(t)\leq m_{1}~\forall t\in\lbrack T_{0},T]\text{.} \label{bd:M1}%
\end{equation}
We remark that $m_{1}$ does not depend on $M_{1}(0)$.
We now derive a differential inequality for $M_{r}$. Using (\ref{CHWEAK}) we
get%
\begin{align*}
\frac{d}{dt}M_{r}(t)  &  =\frac{d}{dt}\left(  \int_{\Omega}\left(
-\ln(u\right)  )^{r}\right)  ^{1/r}\\
&  =-\frac{1}{r}\left(  \int_{\Omega}\left(  -\ln(u\right)  )^{r}\right)
^{1/r-1}\int_{\Omega}r\left\langle u_{t},\frac{\left(  -\ln(u\right)  )^{r-1}%
}{u}\right\rangle \\
&  =\left(  M_{r}\right)  ^{1-r}\int_{\Omega}\nabla\left(  \frac{\left(
-\ln(u)\right)  ^{r-1}}{u}\right)  (\nabla u+\mu\nabla w)\\
&  -\left(  M_{r}\right)  ^{1-r}\int_{\Omega}g\left(  u\right)  \frac{\left(
-\ln(u)\right)  ^{r-1}}{u}\text{.}%
\end{align*}
We focus on the last term only. Using H\"{o}lder's inequality with H\"{o}lder
conjugates $\frac{r}{r-1}$ and $r$ we obtain
\begin{align}
M_{r-1}^{r-1}  &  =\int_{\Omega}(-\ln u)^{r-1}\leq\left(  \int_{\Omega
}1\right)  ^{\frac{1}{r}}\left(  \int_{\Omega}(-\ln u)^{\left(  r-1\right)
\frac{r}{r-1}}\right)  ^{\frac{r-1}{r}}\label{conto}\\
&  =\left\vert \Omega\right\vert ^{\frac{1}{r}}M_{r}^{r-1}\text{.}\nonumber
\end{align}
Hence%
\begin{equation}
M_{r-1}\leq\left\vert \Omega\right\vert ^{\frac{1}{r(r-1)}}M_{r}. \label{g}%
\end{equation}
So%
\begin{align}
-\left(  M_{r}\right)  ^{1-r}\int_{\Omega}g\left(  u\right)  \frac{\left(
-\ln(u)\right)  ^{r-1}}{u}  &  \leq C\left(  M_{r}\right)  ^{1-r}\int_{\Omega
}\left(  -\ln(u)\right)  ^{r-1}\label{stimetta}\\
&  =C\left(  M_{r}\right)  ^{1-r}\left(  M_{r-1}\right)  ^{r-1}\nonumber\\
&  \leq\left\vert \Omega\right\vert ^{\frac{1}{r}}C\left(  M_{r}\right)
^{1-r}\left(  M_{r}\right)  ^{r-1}\leq C\text{,}\nonumber
\end{align}
where $C$ does not depend on $r$. Using (\ref{stimetta}) and proceeding as in
\cite[Lemma 3.1]{LP2}, it is possible to prove the differential inequality%
\[
\frac{d}{dt}M_{r}\leq-C_{3}\frac{1}{r^{2}}\left(  M_{r}\right)  ^{2}%
+C_{4}m_{1}M_{r}+C_{5}r\text{,}%
\]
where $m_{1}$ as in (\ref{bd:M1}) and, hence, the following inequality for
every $\bar{T}\in(0,T]$
\begin{equation}
\sup_{t\geq\bar{T}}M_{r}(t)\leq B_{1}(\bar{T})r^{3}~~\forall r\in
\lbrack1,+\infty), \label{treotto}%
\end{equation}
where $B_{1}(\bar{T})$ is decreasing on $\left(  0,+\infty\right)  $ and such
that $\bar{T}B_{1}(\bar{T})$ is increasing for $\bar{T}$ large enough.
Furthermore $B_{1}(\bar{T})$ does not depend on the initial $M_{r}(0)$ and on
$r$. Proceeding as in \cite[Lemma 3.2 and Lemma 3.3]{LP2},  we can show that
\[
\forall T_{0}>0\text{ }\exists B>0\text{ such that }M_{r}(t)\leq B\text{
}\forall t\geq T_{0}\text{,}%
\]
where $B$ depends on $T_{0}$ and $\bar{u}_{0}$, but not on pointwise values of
$u_{0}$. Passing to the limit $r\rightarrow\infty$ we obtain
\begin{equation}
\left\Vert \ln(u(\cdot,t))\right\Vert _{L^{\infty}(\Omega)}\leq B~\forall
t\in\lbrack T_{0},T] \label{pluto}%
\end{equation}
and so (\ref{A}).

The Proposition \ref{separazione} is proved when (\ref{hp:u0bd}) holds. If
(\ref{hp:u0bd}) is not satisfied, we prove Proposition \ref{separazione} by
approximation: we approximate $u_{0}$ with $u_{0}^{n}$ satisfying
(\ref{hp:u0bd}) and employ the continuous dependence (see Remark
\ref{superRemark}) of solutions to get (\ref{pluto}) even for $u_{0}$ which
does not satisfy (\ref{hp:u0bd})\ (see \cite{LP2}).
\end{proof}

\begin{proposition}
\label{separazione1}Let assumption of Theorem \ref{teorego} be satisfied. Then
for every $T_{0}\in(0,T)$ there exists $k>0$ depending on $T_{0}$ and $\bar
{u}_{0}$ such that
\begin{equation}
u(x,t)\leq1-k\text{ for }a.a.\text{ }x\in\Omega\text{ and }t\in\lbrack
T_{0},T]\,\text{.} \label{B}%
\end{equation}
Furthermore, if there exists $\tilde{k}$ such that%
\begin{equation}
u_{0}\leq1-\tilde{k} \label{hp:u0bd2}%
\end{equation}
then $T_{0}=0$.
\end{proposition}

\begin{proof}
We obtain Proposition \ref{separazione1} from Proposition \ref{separazione}
with $U=1-u$.
\end{proof}

Combining Proposition \ref{separazione} and Proposition \ref{separazione1} we
conclude the proof of Theorem~\ref{teo sep}.

\section{Remarks and generalizations\label{comm}}

\begin{remark}
If the solution to (\ref{eq:CH})-(\ref{eq:ic}) is defined on $[0,+\infty)$
Londen and Petzeltov\`{a} proved in \cite{LP2} that (under the assumptions of
Theorem \ref{teo sep} with $g=0$) $u$ separates from $0$ and $1$ (after
$T_{0}>0$) uniformly in time, i.e. for every $T_{0}>0$ there exists $k>0$ such
that for every $t>T_{0}$ $k\leq u(t)\leq1-k$. We remark that, if $g\neq0$, the
separation properties are not uniform in time even if $g$ satisfies
assumptions of Theorem \ref{teo sep}. Indeed, set $g(u)=-u$. Assumptions
(\ref{G2}), (\ref{ipo:gLip}), (\ref{ipo:g5}) are satisfied. So, for every
$T>0$, there exists an unique $u$ solution to (\ref{eq:CH})-(\ref{eq:ic})
definite over the whole set $[0,T]$. We have already noticed that $\bar{u}%
_{t}=\int_{\Omega}g(u)=-\int_{\Omega}u=-\bar{u}$. So, we have $\bar{u}%
(t)=\bar{u}_{0}\exp(-t)$ and
\begin{equation}
\bar{u}(t)\rightarrow0\text{ for }t\rightarrow+\infty.
\end{equation}
Hence, it is not possible to estimate $k\leq u(t)$ for every $t>T_{0}$ with
$k>0$ not depending on $t$.
\end{remark}

\begin{remark}
It is not hard to prove that our theorems can be obtained also for functions
$g$ that satisfy
\begin{align*}
&  g(x,t,s)\text{ is continuous with respect to }t\text{ and }s \text{ and measurable with respect to }x
\end{align*}
and%
\begin{align*}
&\exists\, C    >0\text{ such that }\left\vert g(x,t,s)\right\vert \leq C\quad \forall t,s    \in\lbrack0,T]\times\lbrack0,1]\text{ and for }a.a.~x\in
\Omega\text{,}%
\end{align*}
instead of (\ref{G2}). Indeed, continuity with respect to $x$ is used only to
ensure (\ref{ipo:Gi3}).
\end{remark}

\begin{remark}
We now remark that assumption (\ref{ipo:g5}) is natural. To the best of the
authors' knowledge, our assumptions on $g$ are satisfied in every work in
which Cahn-Hilliard equation with reaction is studied (see, e. g., \cite{KS},
\cite{BEG}, \cite{BO} or \cite{DP}). Furthermore, suppose that there exists
$c<0$ such that $g(x,t,s)\leq c<0$ for $a.a.$ $(x,t,s)\in\Omega\times
\lbrack0,T]\times\lbrack0,1]$, then it is possible to prove that doesn't exist
$u$ solution to (\ref{eq:CH})-(\ref{eq:ic}) on $[0,T]$ for $T$ large enough.
Indeed, suppose, by contradiction that such a $u$ exists. Then $\bar{u}%
_{t}=\int_{\Omega}g(u)<c|\Omega|<0$, so $\bar{u}(t)\leq\bar{u}_{0}+c|\Omega
|t$. Hence, $\bar{u}(t)<0$ if $t$ is large enough. Furthermore it is possible
to show that such a $t$ can be chosen arbitrary small (if $\bar{u}_{0}$ is
small enough). This argument doesn't prove that our assumptions are sharp, but
shows that they can be considered natural.
\end{remark}

\begin{remark}
Theorem \ref{Teorema principale} can be also extended to the nonlocal
convective Cahn-Hilliard equation with convection%
\[
u_{t}+V\cdot\nabla u+\nabla\cdot J=g(u)
\]
where $V$ denotes the flow speed and $J$ as in (\ref{jei}) (see \cite[Section~6]{FGR}).
\end{remark}

\section{Appendix\label{appendix}}

\subsection{Examples of convolution kernels \label{convoluzione}}

In this Section we provide examples of convolution kernels satisfying
assumptions (\ref{K1})-(\ref{K4}). We prove that
\[
K_{1}(x)=C\exp(-|x|^{2}/\lambda)\text{,}%
\]%
\[
K_{2}(x)=\left\{
\begin{array}
[c]{ll}%
C\exp(\frac{-h^{2}}{h^{2}-|x|^{2}}) & \text{ se }|x|<h\\
0 & \text{ se }|x|\geq h
\end{array}
\right.
\]
and%
\[
\left\{
\begin{array}
[c]{lll}%
K_{3}(\left\vert x\right\vert )=k_{d}\left\vert x\right\vert ^{2-d} &
\text{per } & d>2\\
K_{3}(\left\vert x\right\vert )=-k_{2}\ln\left\vert x\right\vert  & \text{per}
& d=2
\end{array}
\right.  \text{,}%
\]
where $h,\lambda,k_{d}>0$, satisfy (\ref{K1})-(\ref{K4}).

We start considering $K_{1}$ and $K_{2}$. They satisfy (\ref{K1}) trivially.
It is not hard to show that $K_{1}$ and $K_{2}$ are $C^{\infty}\left(
\mathbb{R}
^{d}\right)  \cap W^{2,p}\left(
\mathbb{R}
^{d}\right)  $ for every $1\leq p\leq\infty$. As consequence we have the
estimates (for $i=1,2$)
\[
\int_{\Omega}|K_{i}(x-y)|dy\leq\int_{%
\mathbb{R}
^{d}}|K_{i}(x-y)|dy\leq\int_{%
\mathbb{R}
^{d}}|K_{i}(y)|dy=\left\Vert K_{i}\right\Vert _{L^{1}(%
\mathbb{R}
^{d})}\text{,}%
\]
which yields (\ref{K2}). Set $\rho\in W^{1,p}(\Omega)$, $1\leq p\leq\infty$.
Then
\begin{align}
\int_{\Omega}|K_{i}\ast\rho|^{p}  &  =\int_{\Omega}|\int_{\Omega}%
K_{i}(x-y)\rho(y)dy|^{p}dx=\label{centoquarantasei}\\
&  \leq\int_{\Omega}|\left(  \int_{\Omega}|K_{i}(x-y)|^{\frac{p}{p-1}%
}dy\right)  ^{\frac{p-1}{p}}\left(  \int_{\Omega}|\rho(y)|^{p}dy\right)
^{\frac{1}{p}}|^{p}dx\nonumber\\
&  \leq\left\Vert \rho\right\Vert _{L^{p}(\Omega)}^{p}\int_{\Omega}\left\Vert
K_{i}\right\Vert _{L^{\frac{p}{p-1}}(%
\mathbb{R}
^{d})}^{p-1}dx\nonumber\\
&  =\left\Vert \rho\right\Vert _{L^{p}(\Omega)}^{p}\left\Vert K_{i}\right\Vert
_{L^{\frac{p}{p-1}}(%
\mathbb{R}
^{d})}^{p-1}|\Omega|~\leq C\left\Vert \rho\right\Vert _{L^{p}(\Omega)}%
^{p}\text{,}\nonumber
\end{align}
where $C$ is a positive constant depending on $p$. Since $K_{1}$ and $\ K_{2}$
are $C^{\infty}\left(
\mathbb{R}
^{d}\right)  \cap W^{2,p}\left(
\mathbb{R}
^{d}\right)  $ we have
\begin{align*}
\partial_{j}(K_{i}\ast\rho)  &  =\partial_{j}K_{i}\ast\rho\text{ and}\\
\partial_{jl}(K_{i}\ast\rho)  &  =\partial_{jl}K_{i}\ast\rho\text{ }\forall
j,l\in\{1,...,d\}~i=1,2\text{.}%
\end{align*}
Proceeding as above we get
\begin{align}
\int_{\Omega}|\partial_{j}(K_{i}\ast\rho)|^{p}  &  =\int_{\Omega}|\partial
_{j}K_{i}\ast\rho|^{p}\leq\left\Vert \rho\right\Vert _{L^{p}(\Omega)}%
^{p}\left\Vert \partial_{j}K_{i}\right\Vert _{L^{\frac{p}{p-1}}(%
\mathbb{R}
^{d})}^{p-1}|\Omega|\label{cento47}\\
&  =C\left\Vert \rho\right\Vert _{L^{p}(\Omega)}^{p}~\forall j\in
\{1,...,d\}\nonumber
\end{align}
and%
\begin{align}
\int_{\Omega}|\partial_{jl}(K_{i}\ast\rho)|^{p}  &  =\int_{\Omega}%
|\partial_{jl}K_{i}\ast\rho|^{p}\leq\left\Vert \rho\right\Vert _{L^{p}%
(\Omega)}^{p}\left\Vert \partial_{jl}K_{i}\right\Vert _{L^{\frac{p}{p-1}}(%
\mathbb{R}
^{d})}^{p-1}|\Omega|\label{cento48}\\
&  =C\left\Vert \rho\right\Vert _{L^{p}(\Omega)}^{p}~\forall j\in
\{1,...,d\}\text{,}\nonumber
\end{align}
where $C>0$ depends on $p$. From estimates (\ref{centoquarantasei}) and
(\ref{cento47}) follows (\ref{K3}), and, from (\ref{cento48}), follows
(\ref{K4}).

We now prove that $K_{3}$ satisfies (\ref{K1})-(\ref{K4}). (\ref{K1}) holds
trivially. Property (\ref{K2}) holds thanks to
\begin{align}
\int_{\Omega}\ln|x-y|dy  &  \leq\int_{B_{1}(x)}\ln|x-y|dy+\int_{\Omega
\smallsetminus B_{1}(x)}\ln|x-y|dy\label{rino}\\
&  \leq C+|\Omega|\ln(\max\{diam(\Omega),1\})\text{ for }d=2\nonumber\\
\int_{\Omega}|x-y|^{2-d}dy  &  \leq\int_{B_{1}(x)}|x-y|^{2-d}dy+\int
_{\Omega\smallsetminus B_{1}(x)}|x-y|^{2-d}dy\nonumber\\
&  \leq C+|\Omega|\text{ for }d>2\text{.}\nonumber
\end{align}
In order to prove (\ref{K3}) and (\ref{K4}) we proceed as follows. Since
$\Omega$ is bounded, there exists $R>0$ such that $B_{R}\supseteq\Omega$,
where $B_{R}=\{x\in%
\mathbb{R}
^{d}:|x|<R\}$. Let $A\subset%
\mathbb{R}
^{d}$ be a measurable set and denote
\[
I_{A}(x)=\left\{
\begin{array}
[c]{ll}%
1 & \text{ for }x\in A\\
0 & \text{for }x\notin A
\end{array}
\right.  \text{.}%
\]
We have
\begin{align*}
\left\Vert K_{3}\ast\rho\right\Vert _{L^{p}(\Omega)}^{p}  &  =\int_{\Omega
}\left\vert \int_{\Omega}K_{3}(x-y)\rho(y)dy\right\vert ^{p}dx\\
&  \leq\int_{%
\mathbb{R}
^{d}}I_{B_{R}}(x)\left\vert \int_{%
\mathbb{R}
^{d}}K_{3}(x-y)\rho(y)I_{\Omega}(y)dy\right\vert ^{p}dx\\
&  \leq\int_{%
\mathbb{R}
^{d}}\left\vert \int_{%
\mathbb{R}
^{d}}I_{B_{2R}}(x-y)K_{3}(x-y)\rho(y)I_{\Omega}(y)dy\right\vert ^{p}dx\\
&  =\left\Vert K_{3}^{\prime}\ast\rho^{\prime}\right\Vert _{L^{p}(%
\mathbb{R}
^{d})}^{p}\text{,}%
\end{align*}
where $K_{3}^{\prime}=K_{3}\cdot I_{B_{2R}}$ and $\rho^{\prime}=\rho\cdot
I_{\Omega}$. Using Young's inequality we get, for every $p\in\lbrack
1,+\infty]$,
\begin{equation}
\left\Vert K_{3}^{\prime}\ast\rho^{\prime}\right\Vert _{L^{p}(%
\mathbb{R}
^{d})}\leq\left\Vert K_{3}^{\prime}\right\Vert _{L^{1}(%
\mathbb{R}
^{d})}\left\Vert \rho^{\prime}\right\Vert _{L^{p}(%
\mathbb{R}
^{d})}\text{. } \label{puc}%
\end{equation}
Proceeding as in (\ref{rino}) we obtain
\[
\left\Vert K_{3}^{\prime}\right\Vert _{L^{1}(%
\mathbb{R}
^{d})}=\int_{B_{2R}}\left\vert K_{3}(y)\right\vert dy\leq C\text{.}%
\]
Hence, from (\ref{puc}), we have
\begin{equation}
\left\Vert K_{3}\ast\rho\right\Vert _{L^{p}(\Omega)}\leq C\left\Vert
\rho^{\prime}\right\Vert _{L^{p}(%
\mathbb{R}
^{d})}=C\left\Vert \rho\right\Vert _{L^{p}(\Omega)}\text{.} \label{yuy}%
\end{equation}
Similar computations show%
\begin{equation}
\left\Vert \nabla K_{3}\ast\rho\right\Vert _{L^{p}(\Omega)}=\left\Vert \nabla
K_{3}\right\Vert _{L^{1}(B_{R})}\left\Vert \rho\right\Vert _{L^{p}(\Omega)}.
\label{yuyuy}%
\end{equation}
We remark that, for every $d\geq2$, $\left\vert \nabla K_{3}(x)\right\vert
\leq C_{d}|x|^{1-d}$ where $C_{d}$ denotes a positive constant depending on
$d$. Hence, we get
\begin{align*}
\int_{B_{R}}\left\vert \nabla K_{3}(x-y)\right\vert dy  &  \leq\int_{B_{1}%
(x)}\left\vert \nabla K_{3}(x-y)\right\vert dy+\int_{B_{R}\smallsetminus
B_{1}(x)}\left\vert \nabla K_{3}(x-y)\right\vert dy\\
&  \leq C\int_{B_{1}(x)}|x-y|^{1-d}dy+C\leq C\text{.}%
\end{align*}
So, (\ref{yuy}) and (\ref{yuyuy}) yield%
\begin{align}
\left\Vert K_{3}\ast\rho\right\Vert _{W^{1,p}(\Omega)}  &  \leq C\left(
\left\Vert \nabla K_{3}\ast\rho\right\Vert _{L^{p}(\Omega)}+\left\Vert
K_{3}\ast\rho\right\Vert _{L^{p}(\Omega)}\right) \label{kok}\\
&  \leq C\left\Vert \rho\right\Vert _{L^{p}(\Omega)}\text{.}\nonumber
\end{align}
This proves (\ref{K3}). Property (\ref{K4}) is proved thanks to (\ref{kok}) and
\cite[Theorem~9.9]{GT}.

\subsection{Auxiliary theorems}

\begin{theorem}
\label{Robinson}Let $V\subseteq H\subseteq V^{\ast}$ be an Hilbert tern. Let
$\{u_{n}\}$ be a sequence such that $u_{n}:[0,T]\rightarrow V$ and
\[
\left\Vert u_{n}\right\Vert _{L^{2}(0,T,V)}\leq C\text{, }\left\Vert
u_{n}^{\prime}\right\Vert _{L^{p}(0,T,V^{\ast})}\leq C
\]
where $p>1$ and $C>0$ not depending on $n$. Then there exists a subsequence
$\{u_{n_{k}}\}$ such that
\[
u_{n_{k}}\rightarrow u\text{ in }L^{2}(0,T,H)\text{.}%
\]

\end{theorem}

\begin{proof}
This Theorem is proved in \cite{Ro}, Theorem 8.1.
\end{proof}

\begin{theorem}
\label{teo poinc}Let $\Omega\subset%
\mathbb{R}
^{d}$, $d\in%
\mathbb{N}
$, be a bounded domain with boundary of class $C^{1,1}$. Let $z\in
H^{1}(\Omega)$ and $\Omega_{1}\subseteq\Omega$ such that $|\Omega_{1}|>0$.
Then there exists $C\geq0$ depending on $\Omega$ and $\Omega_{1}$ such that
\begin{equation}
\left\Vert z-\frac{1}{\left\vert \Omega_{1}\right\vert }\int_{\Omega_{1}%
}z\right\Vert _{L^{2}(\Omega)}\leq C\frac{1}{\left\vert \Omega_{1}\right\vert
}\left\Vert \nabla z\right\Vert _{L^{2}(\Omega)\text{.}} \label{dis:Poincare}%
\end{equation}

\end{theorem}

\begin{proof}
This inequality follows from \cite{Zi}, Lemma 4.3.1.
\end{proof}

\begin{theorem}
Let $\Omega\subset%
\mathbb{R}
^{d}$, $d\in%
\mathbb{N}
$, be a bounded domain with boundary of class $C^{1,1}$. Let $z\in
H^{1}(\Omega)$. Then there exists $C\geq0$ depending on $\Omega$ such that
\begin{equation}
\left\Vert z\right\Vert _{L^{2}(\Omega)}^{2}-C\delta^{-d/2}\left\Vert
z\right\Vert _{L^{1}(\Omega)}^{2}\leq\delta\left\Vert \nabla z\right\Vert
_{L^{2}(\Omega)}^{2}\text{ }\forall\delta\in\left(  0,1/2\right)  .
\label{dis:gN}%
\end{equation}

\end{theorem}

\begin{proof}
This inequality is consequence of Gagliardo-Nierenberg interpolation
inequality. A proof can be found in \cite{Ni}, lecture II$.$
\end{proof}

\begin{theorem}
\label{elliptic rego}Let $\Omega\subset%
\mathbb{R}
^{d}$, $d\in%
\mathbb{N}
$, be a bounded domain with boundary of class $C^{1,1}$. Denote with $n$ the
outer unit normal on $\partial\Omega$. Let $\xi\in L^{2}(\Omega)$ and $\eta\in
H^{1/2}(\partial\Omega)$. If $z\in H^{1}(\Omega)$ is weak solution to
\[
\left\{
\begin{array}
[c]{ll}%
\Delta z=\xi & \text{in }\Omega\\
\frac{\partial\Omega}{\partial n}=\eta & \text{on }\partial\Omega
\end{array}
\right.  \text{.}%
\]
Then $z\in H^{2}(\Omega)$. Furthermore there exists $C>0$ not depending on
$\eta$ and $\xi$ such that
\[
\left\Vert z\right\Vert _{H^{2}(\Omega)}\leq C\left(  \left\Vert z\right\Vert
_{L^{2}(\Omega)}+\left\Vert \xi\right\Vert _{L^{2}(\Omega)}+\left\Vert
\eta\right\Vert _{H^{1/2}(\partial\Omega)}\right)  \text{.}%
\]

\end{theorem}

\begin{proof}
This theorem follows from \cite{BG}, Theorem 3.1.5.
\end{proof}

\begin{lemma}
\label{final lemma}Let $V,B,Y$ be Banach spaces such that $V$ is compact
embedded in $B$ and $B$ continuous embedded in $Y$. Let $\phi\in L^{\infty
}(0,T,V)\cap C([0,T],Y)$. Then
\begin{equation}
\phi\in C([0,T],B)\text{.} \label{tyu}%
\end{equation}

\end{lemma}

\begin{proof}
Let $\{s_{n}\}_{n\in%
\mathbb{N}
}\subset\lbrack0,T]$ such that $s_{n}\rightarrow s_{\infty}$ for
$n\rightarrow\infty$. Then $\phi(s_{n})\rightarrow\phi(s_{\infty})$ in $Y$ and
$\{\phi(s_{n})\}$ is bounded in $V$. Thus there exists a subsequence
$s_{n_{k}}$ such that $\phi(s_{n_{k}})$ is convergent in $B$ and thus in $Y.$
Thanks to the uniqueness of the limit, we have $\phi(s_{n_{k}})\rightarrow
\phi(s_{\infty})$ in $B$. Thanks to the arbitrariness of $\{s_{n},s_{\infty
}\}_{n\in%
\mathbb{N}
}$ we have (\ref{tyu}).
\end{proof}

\end{document}